\documentclass[final,leqno]{siamltex704}
\usepackage{amsmath}
\usepackage{amssymb}
\usepackage{graphicx}
\usepackage{tikz}
\usepackage[notcite,notref]{showkeys}
\newtheorem{defi}{Definition}[section]

\newtheorem{algorithm}{Weak Galerkin MFEM Algorithm}

\newcommand{\bq}{{\bf q}}
\newcommand{\bw}{{\bf w}}
\newcommand{\bx}{{\bf x}}
\newcommand{\by}{{\bf y}}

\newcommand{\be}{{\bf e}}

\newcommand{\bv}{{\bf v}}

\def\T{{\mathcal T}}
\def\E{{\mathcal E}}
\def\V{{\mathcal V}}
\def\W{{\mathcal W}}
\def\M{{\mathcal M}}

\def\bn{{\bf n}}
\def\bq{{\bf q}}

\def\ljump{{[\![}}
\def\rjump{{]\!]}}

\def\ljump{{[\![}}
\def\rjump{{]\!]}}

\def\3bar{{|\hspace{-.02in}|\hspace{-.02in}|}}
\def\bbQ{{\mathbb{Q}}}
\def\boldeta{{\boldsymbol\eta}}

\title{A weak Galerkin mixed finite element method for second-order elliptic problems}
\author{Junping Wang\thanks{Division of Mathematical Sciences, National
Science Foundation, Arlington, VA 22230 (jwang@\break nsf.gov). The
research of Wang was supported by the NSF IR/D program, while
working at the Foundation. However, any opinion, finding, and
conclusions or recommendations expressed in this material are those
of the author and do not necessarily reflect the views of the
National Science Foundation.} \and Xiu Ye\thanks{Department of
Mathematics, University of Arkansas at Little Rock, Little Rock, AR
72204 (xxye@ualr.edu). This research was supported in part by
National Science Foundation Grant DMS-1115097.}}
\begin{document}
\maketitle

\begin{abstract}
A new weak Galerkin (WG) method is introduced and analyzed for the
second order elliptic equation formulated as a system of two first
order linear equations. This method, called WG-MFEM, is designed by
using discontinuous piecewise polynomials on finite element
partitions with arbitrary shape of polygons/polyhedra. The WG-MFEM
is capable of providing very accurate numerical approximations for
both the primary and flux variables. Allowing the use of
discontinuous approximating functions on arbitrary shape of
polygons/polyhedra makes the method highly flexible in practical
computation. Optimal order error estimates in both discrete $H^1$
and $L^2$ norms are established for the corresponding weak Galerkin
mixed finite element solutions.
\end{abstract}


\pagestyle{myheadings}

\section{Introduction}\label{Section:Introduction}

Weak Galerkin (WG) refers to a finite element technique for partial
differential equations in which differential operators are
approximated by their weak forms as distributions. In \cite{wy}, a
weak Galerkin method was introduced and analyzed for second order
elliptic equations based on {\em weak gradients}. In this paper, we
shall develop a new weak Galerkin method for second order elliptic
equations formulated as a system of two first order linear
equations. Our model problem seeks a flux function $\bq=\bq(\bx)$
and a scalar function $u=u(\bx)$ defined in an open bounded
polygonal or polyhedral domain $\Omega\subset\mathbb{R}^d\; (d=2,3)$
satisfying
\begin{eqnarray}
\alpha\bq+\nabla u=0,\ \nabla\cdot \bq=f,\quad
\mbox{in}\;\Omega\label{mix}
\end{eqnarray}
and the following Dirichlet boundary condition
\begin{equation}
 u=-g \quad \mbox{on}\; \partial\Omega,\label{bc1}
\end{equation}
where $\alpha=(\alpha_{ij}(\bx))_{d\times d}\in
[L^{\infty}(\Omega)]^{d^2}$ is a symmetric, uniformly positive
definite matrix on the domain $\Omega$. A weak formulation for
(\ref{mix})-(\ref{bc1}) seeks $\bq\in H(div,\Omega)$ and $u\in
L^2(\Omega)$ such that
\begin{eqnarray}
(\alpha\bq,\bv)-(\nabla\cdot\bv,u)&=&\langle g,\bv\cdot\bn\rangle_{\partial\Omega},
\quad\forall\bv\in H(div,\Omega)\label{w-mix1}\\
(\nabla\cdot\bq,w)&=&(f,w),\quad\forall w\in L^2(\Omega).\label{w-mix2}
\end{eqnarray}
Here $L^2(\Omega)$ is the standard space of square integrable
functions on $\Omega$, $\nabla\cdot\bv$ is the divergence of
vector-valued functions $\bv$ on $\Omega$, $H(div,\Omega)$ is the
Sobolev space consisting of vector-valued functions $\bv$ such that
$\bv\in [L^2(\Omega)]^d$ and $\nabla\cdot\bv \in L^2(\Omega)$,
$(\cdot,\cdot)$ stands for the $L^2$-inner product in $L^2(\Omega)$,
and $\langle\cdot,\cdot\rangle_{\partial\Omega}$ is the inner
product in $L^2(\partial\Omega)$.

Galerkin methods based on the weak formulation
(\ref{w-mix1})-(\ref{w-mix2}) and finite dimensional subspaces of
$H(div,\Omega)\times L^2(\Omega)$ with piecewise polynomials are
known as mixed finite element methods (MFEM). MFEMs for
(\ref{mix})-(\ref{bc1}) treat $\bq$ and $u$ as unknown functions and
are capable of providing accurate approximations for both unknowns
\cite{ab,babuska, sue, brezzi,bf,bddf,bdm,rt,wang}. All the existing
MFEMs in literature possess local mass conservation that makes MFEM
a competitive numerical technique in many applications such as oil
reservoir and groundwater flow simulation in porous media. On the
other hand, MFEMs are formulated in subspaces of
$H(div,\Omega)\times L^2(\Omega)$ which requires a certain
continuity of the finite element functions for the flux variable.
More precisely, the flux functions must be sufficiently continuous
so that the usual divergence is well-defined in the classical sense
in $L^2(\Omega)$. This continuity assumption in turn imposes a
strong restriction on the structure of the finite element partition
and the piecewise polynomials defined on them \cite{rt, bdm, bf,
bddf}.

The weak Galerkin method introduced in \cite{wy} (see also
\cite{mwy-wg-stabilization,mwy-wg-stabilization2} for extensions)
was based on a use of {\em weak gradients} in the following
variational formulation: find $u\in H^1(\Omega)$ such that $u=-g$ on
$\partial\Omega$ and
\begin{equation}\label{w1}
(\alpha^{-1}\nabla u, \nabla\phi) = (f,\phi),\qquad \forall \phi\in
H_0^1(\Omega),
\end{equation}
where $H^1(\Omega)$ is the Sobolev space consisting of functions for
which all partial derivatives up to order one are square integrable,
$H_0^1(\Omega)$ is the subspace of $H^1(\Omega)$ consisting of
functions with vanishing value on $\partial\Omega$. Specifically,
the weak Galerkin finite element formulation in
\cite{wy,mwy-wg-stabilization, mwy-wg-stabilization2} can be
obtained from (\ref{w1}) by simply replacing the gradient $\nabla$
by a discrete gradient $\nabla_w$ (it was denoted as $\nabla_d$ in
\cite{wy}) defined by a distributional formula. The discrete
gradient operator $\nabla_w$ is locally-defined on each element. It
has been demonstrated \cite{wy, wy-mw, wy-mz, wy-mgz} that the weak
Galerkin method enjoys an easy-to-implement formulation that is
parameter free and inherits the physical property of mass
conservation locally on each element. Furthermore, the weak Galerkin
method has the flexibility of using discontinuous finite element
functions, as was commonly employed in discontinuous Galerkin and
hybridized discontinuous Galerkin methods \cite{abcm, cgl}.

The goal of this paper is to extend the weak Galerkin method of
\cite{wy,mwy-wg-stabilization, mwy-wg-stabilization2} to the
variational formulation (\ref{w-mix1})-(\ref{w-mix2}) by following
the idea of {\em weak gradients}. It is clear that {\em divergence}
is the principle differential operator in
(\ref{w-mix1})-(\ref{w-mix2}). Thus, an essential part of the
extension is the development of a weakly-defined discrete divergence
operator, denoted by $(\nabla_w\cdot)$, for a class of vector-valued
{\em weak functions} in a finite element setting. Assuming that
there is such a discrete divergence operator $(\nabla_w\cdot)$
defined on a finite element space $\V_h$ for the flux variable
$\bq$, then formally one would have a WG method for
(\ref{mix})-(\ref{bc1}) that seeks $\bq_h\in\V_h$ and $u_h\in\W_h$
satisfying
\begin{eqnarray*}
(\alpha\bq_h,\bv)-(\nabla_w\cdot\bv,u_h)&=&\langle g,\bv\cdot\bn\rangle_{\partial\Omega},\quad\forall\bv\in \V_h\\
(\nabla_w\cdot\bq_h,w)&=&(f,w),\quad\forall w\in \W_h,
\end{eqnarray*}
where $\W_h$ is a properly defined finite element space for the
scalar variable. The rest of the paper will provide details for a
rigorous interpretation and justification of the above formal WG
method, which shall be called WG-MFEM methods.

The WG-MFEM methods have the following features. First of all, the
finite element partition of the domain $\Omega$ is allowed to
consist of arbitrary shape of polygons for $d=2$ and polyhedra for
$d=3$. Secondly, the flux approximation space $\V_h$ consists of two
components, where the first one is given by piecewise polynomials on
each polygon/polyhedra and the second is given by piecewise
polynomials on the edges/faces of the polygon/polyhedra. The second
component shall be used to approximate the normal component of the
flux variable $\bq$ on each edge/face. Furthermore, the scalar
approximation space $\W_h$ consists of piecewise polynomials on each
polygon/polyhedra with one degree higher than that of the flux. For
example, the lowest order of such elements would consist of
piecewise constant for flux and its normal component on each
edge/face plus piecewise linear function for the scalar variable on
each polygon/polyhedra. There is no continuity required for any of
the finite element functions in WG-MFEM.

One close relative of the WG-MFEM is the mimetic finite difference
(MFD) method, see \cite{berndt, bls, veiga-lipnikov-manzini,
veiga-lipnikov-manzini-2} and the reference cited therein. Both
WG-MFEM and MFD share the same flexibility of using
polygonal/polyhedral elements of the domain. The lowest order MFD
approximates the flux by using only piecewise constants on each
edge/face, and the scalar variable by using another piecewise
constant function on each polygonal/polyhedral element, while
WG-MFEM provides a wide class of numerical schemes with arbitrary
order of polynomials. The arbitrary-order mimetic scheme of
\cite{veiga-lipnikov-manzini} is based on the philosophy of
approximating the forms in (\ref{w1}) using nodal-based polynomials,
which has very minimal overlap with the WG-MFEM. It should be
pointed out that there are other numerical methods designed on
general polygonal meshes in existing literature \cite{causin-sacco,
droniou-eymard, abcm, cgl, herbin, halama}. In particular, a
comparison with existing numerical schemes should be conducted on a
certain set of benchmark problems as demonstrated in \cite{herbin}.
More benchmarks can be found in \cite{halama}.

Allowing arbitrary shape for mesh elements provides a convenient
flexibility in both numerical approximation and mesh generation,
especially in regions where the domain geometry is complex. Such a
flexibility is also very much appreciated in adaptive mesh
refinement methods. This is a highly desirable feature in practice
since a single type of mesh technology is too restrictive in
resolving complex multi-dimensional and multi-scale problems
efficiently \cite{ka}.

The paper is organized as follows. In Section \ref{section2}, we
present a discussion of weak divergence operator in some
weakly-defined spaces. In Section \ref{wg-mfem}, we provide a
detailed description and assumptions for the WG-MFEM method. In
Section \ref{section-projections}, we define some local projection
operators and then derive some approximation properties which are
useful in error analysis. In Section \ref{section-erroranalysis}, we
shall establish an optimal order error estimate for the WG-MFEM
approximations in a norm that is related to the $L^2$ for the flux
and $H^1$ for the scalar function. In Section
\ref{error-analysis-ell2}, we derive an optimal order error estimate
in $L^2$ for the scalar approximation by using a duality argument as
was commonly employed in the standard Galerkin finite element
methods \cite{ci, sue}. Finally, we provide some technical results
in the appendix that are critical in dealing with finite element
functions on arbitrary polygons/polyhedra.

\section{Weak Divergence}\label{section2}

The key in weak Galerkin methods is the use of weak derivatives in the place
of strong derivatives in the variational form for the underlying
partial differential equations. For the mixed problem (\ref{mix})
with boundary condition (\ref{bc1}), the corresponding variational
form is given by (\ref{w-mix1}) and (\ref{w-mix2}), where divergence
is the only differential operator involved in the formulation. Thus,
understanding weak divergence is critically important in the
corresponding WG method. The goal of this section is to introduce a
weak divergence operator and its approximation by using piecewise
polynomials.

Let $K$ be any polygonal or polyhedral domain with interior $K^0$
and boundary $\partial K$. A {\em weak vector-valued function} on
the region $K$ refers to a vector-valued function $\bv=\{\bv_0,
\bv_b\}$ such that $\bv_0\in [L^2(K)]^d$ and $\bv_b\cdot\bn\in
H^{-\frac12}(\partial K)$, where $\bn$ is the outward normal
direction of $K$ on its boundary. The first component $\bv_0$ can be
understood as the value of $\bv$ in the interior of $K$, and the
second component $\bv_b$ represents $\bv$ on the boundary of $K$.
The requirement of $\bv_b\cdot\bn\in H^{-\frac12}(\partial K)$
indicates that we are merely interested in the normal component of
$\bv_b$. Note that $\bv_b$ may not be necessarily related to the
trace of $\bv_0$ on $\partial K$ should a trace be well defined.
Denote by $\M(K)$ the space of weak vector-valued functions on $K$;
i.e.,
\begin{equation}\label{hi.888}
\M(K) = \{\bv=\{\bv_0, \bv_b \}:\ \bv_0\in [L^2(K)]^d,\;
\bv_b\cdot\bn\in H^{-\frac12}(\partial K)\}.
\end{equation}
Following the definition of weak gradient introduced in \cite{wy},
we define a {\em weak divergence operator} as follows.
\medskip

\begin{defi}
The dual of $L^2(K)$ can be identified with itself by using the
standard $L^2$ inner product as the action of linear functionals.
With a similar interpretation, for any $\bv\in \M(K)$, the {\em weak
divergence} of $\bv$ is defined as a linear functional $\nabla_w
\cdot\bv$ in the dual space of $H^1(K)$ whose action on each
$\varphi\in H^1(K)$ is given by
\begin{equation}\label{weak-divergence}
(\nabla_w\cdot\bv, \varphi)_K := -(\bv_0, \nabla\varphi)_K + \langle
\bv_b\cdot\bn, \varphi\rangle_{\partial K},
\end{equation}
where $\bn$ is the outward normal direction to $\partial K$,
$(\bv_0,\nabla\varphi)_K=\int_K \bv_0\cdot\nabla\varphi dK$ is the
action of $\bv_0$ on $\nabla\varphi$, and $\langle \bv_b\cdot\bn,
\varphi\rangle_{\partial K}$ is the action of $\bv_b\cdot\bn$ on
$\varphi\in H^{\frac12}(\partial K)$.
\end{defi}

\medskip

The Sobolev space $[H^1(K)]^d$ can be embedded into the space
$\M(K)$ by an inclusion map $i_\M: \ [H^1(K)]^d\to \M(K)$ defined as
follows
$$
i_\M(\bq) = \{\bq|_{K}, \bq|_{\partial K}\}.
$$
With the help of the inclusion map $i_\M$, the Sobolev space
$[H^1(K)]^d$ can be viewed as a subspace of $\M(K)$ by identifying
each $\bq\in [H^1(K)]^d$ with $i_\M(\bq)$. Analogously, a weak
vector-valued function $\bv=\{\bv_0,\bv_b\}\in \M(K)$ is said to be
in $[H^1(K)]^d$ if it can be identified with a function $\bq\in
[H^1(K)]^d$ through the above inclusion map. It is not hard to see
that $\nabla_w\cdot\bv=\nabla\cdot\bv$ if $\bv$ is a smooth function
in $[H^1(K)]^d$.

Next, we introduce a discrete weak divergence operator by
approximating $\left(\nabla_w\ \cdot\ \right)$ in a polynomial
subspace of the dual of $H^1(K)$. To this end, for any non-negative
integer $r\ge 0$, denote by $P_{r}(K)$ the set of polynomials on $K$
with degree no more than $r$. A discrete weak divergence operator,
denoted by $(\nabla_{w,r}\cdot)$, is defined as the unique
polynomial $(\nabla_{w,r}\cdot\bv) \in P_r(K)$ that satisfies the
following equation
\begin{equation}\label{discrete-weak-divergence-element}
(\nabla_{w,r}\cdot\bv, \phi)_K = -(\bv_0,\nabla\phi)_K+
\langle\bv_b\cdot\bn,  \phi\rangle_{\partial K},\qquad \forall
\phi\in P_r(K).
\end{equation}

\section{Weak Galerkin MFEM: Assumptions and Algorithm}\label{wg-mfem}

Let ${\cal T}_h$ be a partition of the domain $\Omega$ consisting of
polygons in 2D and polyhedra in 3D. Denote by ${\cal E}_h$ the set
of all edges or flat faces in ${\cal T}_h$, and let ${\cal
E}_h^0={\cal E}_h\backslash\partial\Omega$ be the set of all
interior edges or flat faces. For every element $T\in \T_h$, we
denote by $|T|$ the area or volume of $T$ and by $h_T$ its diameter.
Similarly, we denote by $|e|$ the length or area of $e$ and by $h_e$
the diameter of edge or flat face $e\in\E_h$. We also set as usual
the mesh size of $\T_h$ by
$$
h=\max_{T\in\T_h} h_T.
$$

\begin{figure}[h!]
\begin{center}
\begin{tikzpicture}
\coordinate (A) at (-3,0); \coordinate (B) at (2,-2); \coordinate
(C) at (3.5, 0.5); \coordinate (D) at (1,1); \coordinate (E) at
(2.2, 3.3); \coordinate (F) at (-1, 3); \coordinate (CC) at (0,0);
\coordinate (Ae) at (1.2,-0.5); \coordinate (xe) at (-1.5, 2.25);
\coordinate (AFc) at (-2, 1.5); \coordinate (AFcLeft) at (-2.5,
1.83333); \draw node[left] at (xe) {$\bx_e$}; \draw node[right] at
(Ae) {$A_e$}; \draw node[left] at (A) {A}; \draw node[below] at (B)
{B}; \draw node[right] at (C) {C}; \draw node[left] at (D) {D};
\draw node[above] at (E) {E}; \draw node[above] at (F) {F}; \draw
node[left] at (AFcLeft) {$\bn$}; \draw
(A)--(B)--(C)--(D)--(E)--(F)--cycle; \draw[dashed](A)--(Ae);
\draw[dashed](F)--(Ae); \draw[->] (Ae)--(xe); \draw[->]
(AFc)--(AFcLeft); \filldraw[black] (A) circle(0.05);
    \filldraw[black] (F) circle(0.05);
    \filldraw[black] (Ae) circle(0.05);
    \filldraw[black] (xe) circle(0.035);
\end{tikzpicture}
\caption{Depiction of a shape-regular polygonal element $ABCDEFA$.}
\label{fig:shape-regular-element}
\end{center}
\end{figure}
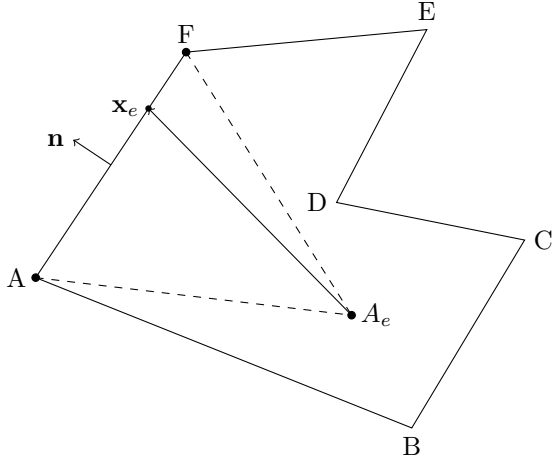

All the elements of $\T_h$ are assumed to be closed and simply
connected polygons or polyhedra, see Fig.
\ref{fig:shape-regular-element}. We need some shape regularity for
the partition $\T_h$ described as below.

\medskip
\begin{description}
\item[A1:] \ Assume that there exist two positive constants $\varrho_v$ and $\varrho_e$
such that for every element $T\in\T_h$ we have
\begin{equation}\label{a1}
\varrho_v h_T^d\leq |T|,\qquad \varrho_e h_e^{d-1}\leq |e|
\end{equation}
for all edges or flat faces of $T$.

\item[A2:] \ Assume that there exists a positive constant $\kappa$ such that for every element
$T\in\T_h$ we have
\begin{equation}\label{a2}
\kappa h_T\leq h_e
\end{equation}
for all edges or flat faces $e$ of $T$.

\item[A3:] \ Assume that the mesh edges or faces are flat. We
further assume that for every $T\in\T_h$, and for every edge/face
$e\in \partial T$, there exists a pyramid $P(e,T, A_e)$ contained in
$T$ such that its base is identical with $e$, its apex is $A_e\in
T$, and its height is proportional to $h_T$ with a proportionality
constant $\sigma_e$ bounded away from a fixed positive number
$\sigma^*$ from below. In other words, the height of the pyramid is
given by $\sigma_e h_T$ such that $\sigma_e\ge \sigma^*>0$. The
pyramid is also assumed to stand up above the base $e$ in the sense
that the angle between the vector $\bx_e-A_e$, for any $x_e\in e$,
and the outward normal direction of $e$ (i.e., the vector $\bn$ in
Fig. \ref{fig:shape-regular-element}) is strictly acute by falling
into an interval $[0, \theta_0]$ with $\theta_0< \frac{\pi}{2}$.

\item[A4:] \ Assume that each $T\in\T_h$ has a circumscribed simplex $S(T)$ that is
shape regular and has a diameter $h_{S(T)}$ proportional to the
diameter of $T$; i.e., $h_{S(T)}\leq \gamma_* h_T$ with a constant
$\gamma_*$ independent of $T$. Furthermore, assume that each
circumscribed simplex $S(T)$ intersects with only a fixed and small
number of such simplices for all other elements $T\in\T_h$.
\end{description}

\medskip
Figure \ref{fig:shape-regular-element} depicts a polygonal element
that is shape regular. Reader are referred to \cite{bls} for a
similar, but different type of shape regularity assumption for the
underlying finite element partition of the domain. The shape
regularity is required for deriving error estimates for locally
defined projection operators to be detailed in later sections.

\medskip
Recall that, on each element $T\in\T_h$, there is a space of weak
vector-valued functions $\M(T)$ defined as in (\ref{hi.188}). Denote
by $\M$ the weak vector-valued function space on $\T_h$ given by
\begin{equation}\label{vspace}
 \M:=\prod_{T\in\T_h}\M(T).
\end{equation}
Note that functions in $\M$ are defined on each element, and there
are two sided values of $\bv_b$ on each interior edge/face $e$,
depicted as $\bv_b|_{\partial T_1}$ and $\bv_b|_{\partial T_2}$ in
Fig. \ref{fig:triangle-11}.

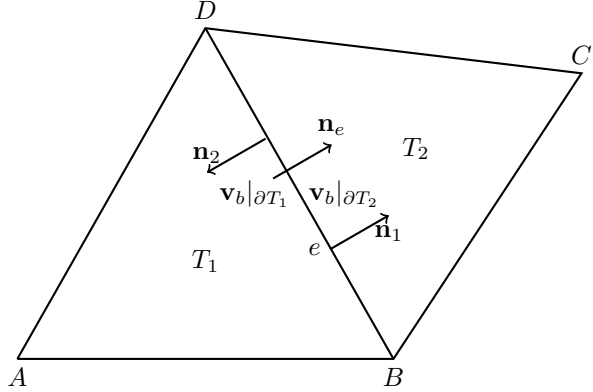
\begin{figure}[h]
\begin{center}
\begin{tikzpicture}
    \path (0,0) coordinate (A1);
    \path (5,0) coordinate (A2);
    \path(7.5,3.8) coordinate (A3);
    \path (2.5, 4.4) coordinate (A4);
    \path (3.4, 2.4) coordinate (A24);
    \path (3.75, 2.2) coordinate (A24C);
    \path (4.16, 1.46) coordinate (A24Low);
    \path (3.3, 2.93) coordinate (A24Up);
    \path (A24) ++(30:0.9cm) coordinate (A24To);
    \path (A24Low) ++(30:0.9cm) coordinate (A24LowTo);
    \path (A24Up) ++(210:0.9cm) coordinate (A24UpTo);
    \path (2.5, 1.3) coordinate (T1);
    \path (5.3,2.8) coordinate (T2);
    \draw [thick] (A1) -- (A2) -- (A3) -- (A4) -- (A1);
    \draw [thick] (A2) -- (A4);
    \draw[->,thick] (A24) -- (A24To) node[above]{$\bn_e$};
    \draw[->,thick] (A24Low) -- (A24LowTo) node[below]{$\bn_1$};
    \draw[->,thick] (A24Up) -- (A24UpTo) node[above]{$\bn_2$};
    \draw node[left] at (A24Low) {$e$};
    \draw node at (T1) {$T_1$};
    \draw node at (T2) {$T_2$};
    \draw node[below] at (A1) {$A$};
    \draw node[below] at (A2) {$B$};
    \draw node[above] at (A3) {$C$};
    \draw node[above] at (A4) {$D$};
    \draw node[left] at (A24C) {$\bv_b|_{\partial T_1}$};
    \draw node[right] at (A24C) {$\bv_b|_{\partial T_2}$};
\end{tikzpicture}
\end{center}
\caption{An interior edge shared by two elements}
\label{fig:triangle-11}
\end{figure}
Let $\V\subset\M$ be a subspace of $\M$ consisting of weak
vector-valued functions which are continuous across each interior
edge/face $e$ in the normal direction; i.e.,
\begin{equation}\label{thanks.100}
\V= \left\{\bv\in\M: \ \exists\ \bq\in H(div,\Omega)\ s.t. \
\bq|_{\partial T}\cdot\bn = (\bv_b)_{\partial T}\cdot\bn, \ \forall
T\in \T_h\right\}.
\end{equation}
The weak divergence operator as defined in (\ref{weak-divergence})
can be extended to any weak vector-valued function $\bv\in \V$ by
taking weak divergence locally on each element $T$. More precisely,
the weak divergence of any $\bv\in\V$ is defined element-by-element
as follows:
$$
\nabla_w\cdot \bv= \nabla_w\cdot(\bv|_{T}),\qquad \mbox{on } \
T\in\T_h.
$$
Similarly, the discrete weak divergence as defined in
(\ref{discrete-weak-divergence-element}) can be extended to $\V$ by
defining
\begin{equation}\label{discrete-weak-divergence-global}
\nabla_{w,r}\cdot \bv = \nabla_{w,r}\cdot(\bv|_{T}),\qquad \mbox{on
} \ T\in\T_h.
\end{equation}

\medskip
The definition of weak divergence of $\bv=\{\bv_0,\bv_b\}\in\V$
requires the value of $\bv$ on each element $T$, namely $\bv_0$, and
the normal component of $\bv_b$ on each edge or face $e\in\E_h$.
Thus, it is the normal component of $\bv$ on each $e\in\E_h$ that
really enters into the equation of discussion in numerical methods.
For convenience, we introduce a set of normal directions on ${\cal
E}_h$ as follows
\begin{equation}\label{thanks.101}
{\cal D}_h = \{\bn_e: \mbox{ $\bn_e$ is unit and normal to $e$},\
e\in {\cal E}_h \}.
\end{equation}
Figure \ref{fig:triangle-11} shows one example of $\bn_e$ pointing
from the element $T_1$ to $T_2$. In the rest of this paper, we will
be concerned with a subspace of $\V$ in which the second component
of $\bv=\{\bv_0, \bv_b\}$ represents the normal component of $\bv$
on each $e\in \E_h$; i.e., $(\bv_b)|_{e}=(\bv|_{e}\cdot\bn_e)\bn_e$.

A discrete weak vector-valued function $\bv=\{\bv_0,\;\bv_b\}$
refers to weak vector-valued functions in which both $\bv_0$ and
$\bv_b$ are vector-valued polynomials. Since the second component is
represented as $\bv_b=v_b \bn_e$ where $\bn_e$ is the prescribed
normal direction to $e\in{\cal E}_h$, then $v_b$ is required to be a
polynomial on each edge or flat face of $\partial T$. Recall that,
for each element $T\in {\cal T}_h$, $P_k(T)$ denotes the set of
polynomials on $T$ with degree no more than $k\ge 0$ and $P_\ell(e)$
is the set of polynomials on $e\in\E_h$ with degree no more than
$\ell\ge 0$.

We now introduce two finite element spaces which are necessary for
formulating our numerical schemes. The first one corresponds to the
scalar (or pressure) variable defined as follows
\[
\W_h=\{w\in L^2(\Omega);\;w|_T\in P_{k+1}(T)\},
\]
where $k\ge 0$ is a non-negative integer. The second one corresponds
to vector-valued functions and their normal components on the set
${\cal E}_h$ of edges or flat faces, and is given by
\begin{equation}\label{vh}
\V_h=\left\{ \bv=\{\bv_0, \bv_b\}:\; \bv_0|_T\in [P_k(T)]^d,
\bv_b|_e=v_b\bn_e, v_b\in P_k(e),\;e\in {\cal E}_h\right\}.
\end{equation}
The pair $\V_h\times \W_h$ forms a finite element approximation
space for the unknowns $\bq$ and $u$ of the problem
(\ref{w-mix1})-(\ref{w-mix2}). For simplicity of notation and
discussion, we shall refer the above defined finite element spaces
as $([P_k(T)]^d,P_k(e), P_{k+1}(T))$ element. The lowest order of
such element makes use of piecewise constant for the flux variable
on each element $T$ and its edges/faces and piecewise linear for the
pressure (scalar) variable.

The discrete weak divergence $(\nabla_{w,k+1}\cdot)$ as defined in
(\ref{discrete-weak-divergence-global}) and
(\ref{discrete-weak-divergence-element}) then provides a linear map
from the finite element space $\V_h$ to $\W_h$. In particular, for
any $\bv\in \V_h$ and $w\in \W_h$, we have the following relation
\begin{eqnarray}\nonumber
(\nabla_{w,k+1}\cdot\bv, w): &=& \sum_{T\in {\cal T}_h}
(\nabla_{w,k+1}\cdot\bv, w)_T\\
&=& \sum_{T\in\T_h} \left( \int_{\partial T} (\bv_b\cdot\bn ) w ds -
\int_T \bv_0\cdot (\nabla w) dT\right).\label{tg.01}
\end{eqnarray}
With an abuse of notation, we shall use $(\nabla_w\cdot)$ to denote
the discrete weak divergence operator $(\nabla_{w,k+1}\cdot)$ in the
rest of this paper.

With the discrete divergence given by (\ref{tg.01}), one might
naively formulate a finite element method by using (\ref{w-mix1})
and (\ref{w-mix2}) as follows. Find $\bq_h=\{\bq_{0},\bq_{b}\}\in
\V_h$ and $u_h\in \W_h$ such that
\begin{eqnarray}
(\alpha\bq_{0},\bv_0)-(\nabla_w\cdot\bv,u_h)&=&\langle g,
\bv_b\cdot\bn\rangle_{\partial\Omega},\quad\forall\bv=\{\bv_0,\bv_b\}\in \V_h\label{nw-mix1}\\
(\nabla_w\cdot\bq_h,w)&=&(f,w),\quad\forall w\in
\W_h.\label{nw-mix2}
\end{eqnarray}
Unfortunately, due to an insufficient enforcement on the component
$\bq_{b}$, the resulting system of linear equations from
(\ref{nw-mix1})-(\ref{nw-mix2}) generally does not have a unique
solution. One remedy to this problem is to stabilize the bilinear
form $(\alpha\bq_{0},\bv_0)$ by requiring some communication between
$\bq_{0}$ and $\bq_{b}$. To this end, we introduce three bilinear
forms as follows.
\begin{eqnarray}
a(\boldeta,\bv) & = & (\alpha\boldeta_0,\bv_0),\qquad \boldeta,\bv\in \V_h,\label{a}\\
b(\boldeta,w) & = & (\nabla_w\cdot\boldeta,\;w),\qquad \boldeta\in \V_h, w\in \W_h, \label{b} \\
s(\boldeta,\bv) & = & \rho\sum_{T\in {\cal T}_h}
h_T\langle(\boldeta_0-\boldeta_b)\cdot\bn,\;(\bv_0-\bv_b)\cdot\bn
\rangle_{\partial T}, \qquad \boldeta,\bv\in \V_h,\label{c}
\end{eqnarray}
and a stabilized bilinear form $a_s(\cdot,\cdot)$:
$$
a_s(\boldeta,\bv):=a(\boldeta,\bv)+s(\boldeta,\bv).
$$
Here $\rho>0$ is any parameter and $h_T$ is the size of $T$. In
practical computation, one might chose $\rho=1$ and substitute $h_T$
by the mesh size $h$ for quasi-uniform partitions; i.e., partitions
for which $h_T/h$ is bounded from below and above uniformly in $T$.

\medskip

\begin{algorithm} Let $\T_h$ be a shape regular finite element
partition of $\Omega$ and $\V_h\times\W_h$ be the corresponding
finite element spaces consisting of $([P_k(T)]^d,P_k(e),
P_{k+1}(T))$ elements with $k\ge 0$. An approximation for
(\ref{mix})-(\ref{bc1}) is given by seeking
$\bq_h=\{\bq_0,\bq_b\}\in \V_h$ and $u_h\in \W_h$ such that
\begin{eqnarray}
a_s(\bq_h,\bv)-b(\bv, u_h)&=&\langle g,\;\bv_b\cdot\bn\rangle_{\partial\Omega},\label{wg1}\\
b(\bq_h, w)&=&(f,\;w),\label{wg2}
\end{eqnarray}
for any $\bv=\{\bv_0, \bv_b\}\in \V_h$ and $w\in \W_h$. The pair of
solutions $(\bq_h;u_h)$ is called a weak Galerkin mixed finite
element approximation of (\ref{mix})-(\ref{bc1}).
\end{algorithm}

\smallskip

The WG-MFEM scheme (\ref{wg1})-(\ref{wg2}) retains the mass
conservation property of (\ref{mix})-(\ref{bc1}) locally on each
element. This can be seen by choosing a piecewise constant test
function $w$ in (\ref{wg2}). More precisely, for any element $T\in
\T_h$, let $w$ assume the value $1$ on $T$ and zero elsewhere. It
follows from (\ref{wg2}) that
$$
(\nabla_w\cdot\bq_h, 1)_T = (f, 1)_T.
$$
The definition of weak gradient (\ref{tg.01}) then implies the
following identity
$$
\int_{\partial T} \bq_b\cdot\bn ds = \int_T f dx,
$$
which asserts a mass conservation for the WG-MFEM method.

The rest of the paper shall provide a theoretical foundation for the
solvability and accuracy of the weak Galerkin mixed finite element
scheme (\ref{wg1})-(\ref{wg2}).

\section{Local Projections: Definition and
Properties}\label{section-projections}

We introduce two projection operators into the finite element spaces
$\V_h$ and $\W_h$ by using local $L^2$ projections. To this end, for
each element $T\in \T_h$, denote by $Q_{0}$ the usual $L^2$
projection from $L^2(T)$ onto $P_k(T)$. Similarly, for each edge or
flat face $e\in {\cal E}_h$, let $Q_{b}$ be the $L^2$ projection
from $L^2(e)$ onto $P_k(e)$. For any $\bv=\{\bv_0, v_b\bn_e\} \in
\V$ with $v_b\in L^2(e)$ on each edge/face $e$, denote by
$$
Q_h\bv:=\{Q_0(\bv_0),\;Q_b(v_b)\bn_e\}
$$
the $L^2$ projection of $\bv$ in $\V_h$. Next, denote by $\bbQ_h$
the $L^2$ projection from $L^2(\Omega)$ onto $\W_h$. Observe that
$\bbQ_h$ is in fact a composition of locally defined $L^2$
projections into the polynomial space $P_{k+1}(T)$ for each element
$T\in {\cal T}_h$.

In the usual finite element error analysis, one often reduces the
error for finite element solutions into the error between the exact
solution and an appropriately defined local projection or
interpolation of the solution. For the WG-MFEM method discussed in
previous sections, this refers to the error between the exact
solution and its local $L^2$ projection. The difficulty in
estimating the projection error arises from the fact that the finite
element partition $\T_h$ contains arbitrary polygons or polyhedra
that are different from the usual simplices as commonly employed in
the standard finite element methods \cite{ci}.

For simplicity of notation, we shall use $\lesssim$ to denote less
than or equal to up to a constant independent of the mesh size,
variables, or other parameters appearing in the inequality.

\begin{lemma}\label{L2errorbound}
Let $\T_h$ be a finite element partition of $\Omega$ satisfying the
shape regularity assumptions {\bf A1 - A4} as given in Section
\ref{wg-mfem}. Then, we have
\begin{eqnarray}
&\sum_{T\in\T_h} \|\bq-Q_0\bq\|_{T}^2 \lesssim h^{2(s+1)}
\|\bq\|_{s+1}^2,\quad 0\le s\le k\label{Jone-term}\\
&\sum_{T\in\T_h} \|\nabla(\bq-Q_0\bq)\|_{T}^2 \lesssim h^{2s}
\|\bq\|_{s+1}^2,\quad 0\le s\le k\label{Jtwo-term}\\
&\sum_{T\in\T_h}
\left(\|u-\bbQ_hu\|_T^2+h_T^2\|\nabla(u-\bbQ_hu)\|_T^2\right)
\lesssim h^{2(s+1)}\|u\|_{s+1}^2,\quad 0\le s \le
k+1.\label{Jthree-term}
\end{eqnarray}
\end{lemma}
\begin{proof}
To derive (\ref{Jone-term}), let $S(T)$ be the circumscribed simplex
of $T$ on which $\bq$ can be defined by smooth extension if
necessary. Let $\tilde Q_0\bq$ be the projection of $\bq$ in the
element defined on $S(T)$. It follows that
\begin{equation}\label{aaa-new.99}
\|\bq-Q_0\bq\|_{T} \leq \|\bq-\tilde Q_0\bq\|_{T}\leq \|\bq-\tilde
Q_0\bq\|_{S(T)}\lesssim h_T^{s+1}\|\bq\|_{s+1,S(T)}
\end{equation}
Using the above estimate we obtain
\begin{equation}\label{www.08}
\sum_{T\in\T_h} \|\bq-Q_0\bq\|_{T}^2 \lesssim
h^{2(s+1)}\sum_{T\in\T_h} \|\bq\|_{s+1, S(T)}^2.
\end{equation}
It follows from the assumption {\bf A4} that the set of the
circumscribed simplices $\{S(T):\ T\in\T_h\}$ has a fixed and small
number of overlaps. Thus, the following estimate holds true
$$
\sum_{T\in\T_h} \|\bq\|_{s+1, S(T)}^2\lesssim \|\bq\|_{s+1}^2.
$$
Substituting the above inequality into (\ref{www.08}) yields the desired
estimate (\ref{Jone-term}).

To derive (\ref{Jtwo-term}), we use the triangle inequality and the
standard error estimate on $S(T)$ to obtain
\begin{eqnarray}
\|\nabla(\bq-Q_0\bq)\|_T &\leq & \|\nabla(\bq-\tilde Q_0\bq)\|_{S(T)}
+\|\nabla(\tilde Q_0\bq-Q_0\bq)\|_{S(T)}\nonumber\\
&\le & Ch_T^s\|\bq\|_{s+1,S(T)} + Ch_T^{-1}\|\tilde
Q_0\bq-Q_0\bq\|_{S(T)},\label{aaa-new.88}
\end{eqnarray}
where we have also used the standard inverse inequality in the
second line. Notice that the assumption {\bf A3} on $\T_h$ implies
that there exists a ball $B\subset T$ with a diameter proportional
to $h_T$. Thus, we have from Lemma \ref{appendix.thm} (see Appendix)
that
\begin{eqnarray*}
\|\tilde Q_0\bq-Q_0\bq\|_{S(T)}&&\lesssim \|\tilde Q_0\bq-Q_0\bq\|_{B}
\le \|\tilde Q_0\bq-Q_0\bq\|_{T}\\
&& \le \|\bq-\tilde Q_0\bq\|_{T} + \|\bq-Q_0\bq\|_{T}\\
&&\lesssim h_T^{s+1}\|\bq\|_{s+1,S(T)}.
\end{eqnarray*}
Substituting the above estimate into (\ref{aaa-new.88}) yields
$$
\|\nabla(\bq-Q_0\bq)\|_T\lesssim h_T^s\|\bq\|_{s+1,S(T)}.
$$
Summing up the above estimate over $T\in\T_h$ leads to the desired
estimate (\ref{Jtwo-term}). Finally, the estimate
(\ref{Jthree-term}) can be established analogously to
(\ref{Jone-term}) and (\ref{Jtwo-term}).
\end{proof}

\medskip
We shall derive two equations that play useful roles in the error
analysis for the WG-MFEM. The first equation is given by the
following Lemma.
\begin{lemma}\label{Lemma4-1}
For any $\bq\in [H^1(\Omega)]^d$, let $Q_h\bq\in \V_h$ be the
projection given by local $L^2$ projections. Then, on each element
$T$, we have
\begin{eqnarray}
(\nabla_w\cdot (Q_h\bq),\;w)_T=(\nabla\cdot\bq,
w)_T-\langle\bq\cdot\bn-Q_b(\bq\cdot\bn), w\rangle_{\partial
T}\label{key.before}
\end{eqnarray}
for all $w\in P_{k+1}(T)$. Moreover, by summing (\ref{key.before})
over all $T\in\T_h$ we obtain
\begin{eqnarray}
b(Q_h\bq,w)=(\nabla\cdot\bq,
w)-\sum_{T\in\T_h}\langle\bq\cdot\bn-Q_b(\bq\cdot\bn),
w\rangle_{\partial T}.\label{key}
\end{eqnarray}

\end{lemma}
\begin{proof}
For $\bq\in [H^1(\Omega)]^d$, we apply the definition of $Q_h$ and
the weak divergence $\nabla_w\cdot (Q_h\bq)$ to obtain
\begin{eqnarray}
(\nabla_w\cdot Q_h\bq,\;w)_T&=&-(Q_0\bq,\;\nabla w)_T+\langle
Q_b(\bq\cdot\bn_e)\bn_e\cdot\bn,\; w\rangle_{\partial T}\nonumber\\
&=&-(Q_0\bq,\;\nabla w)_T+\langle Q_b(\bq\cdot\bn),\;
w\rangle_{\partial T}\label{thanksgiving.201}
\end{eqnarray}
for all $w\in P_{k+1}(T)$. Here we have used the fact that
$\bn_e=\pm\bn$. Since $Q_0$ is the $L^2$ projection into
$[P_k(T)]^d$, then
\begin{eqnarray}
(Q_0\bq,\;\nabla w)_T &=& (\bq, \nabla w)_T\nonumber\\
&=& -(\nabla\cdot\bq, w)_T + \langle \bq\cdot\bn, w\rangle_{\partial
T},\label{thanksgiving.200}
\end{eqnarray}
where we have applied the usual divergence theorem in the second
line. Substituting (\ref{thanksgiving.200}) into
(\ref{thanksgiving.201}) yields (\ref{key.before}). This completes
the proof of the lemma.
\end{proof}
\medskip

The second equation is concerned with the bilinear form
$(\nabla_w\cdot\bv, \bbQ_{h}w)$ for $\bv\in \V_h$ and $w\in
H^1(\Omega)$. Using the definition of $\bbQ_h$ and the integration
by parts, we have for any $\bv=\{\bv_0, \bv_b \}\in \V_h$ and $w\in
H^1(T)$ that
\begin{eqnarray*}
(\nabla_w\cdot \bv,\; \bbQ_hw)_T &=& -(\bv_0,\; \nabla (\bbQ_hw))_T
+ \langle \bv_b\cdot\bn,\;
\bbQ_hw\rangle_{\partial T}\nonumber\\
&=& (\nabla\cdot \bv_0,\; \bbQ_hw)_T + \langle (\bv_b - \bv_0)\cdot
\bn,\;
\bbQ_hw\rangle_{\partial T}\nonumber\\
&=& (\nabla\cdot \bv_0,\; w)_T + \langle (\bv_b - \bv_0)\cdot \bn,\;
\bbQ_hw\rangle_{\partial T}\nonumber\\
&=& -(\bv_0,\; \nabla w)_T + \langle \bv_0\cdot\bn,\;
w\rangle_{\partial T}+\langle (\bv_b - \bv_0)\cdot \bn,\; \bbQ_hw\rangle_{\partial T}\nonumber\\
&=& -(\bv_0,\; \nabla w)_T + \langle (\bv_0-\bv_b)\cdot \bn,\;
w-\bbQ_hw\rangle_{\partial T} + \langle \bv_b\cdot\bn,\;
w\rangle_{\partial T}.
\end{eqnarray*}
The result can be summarized as follows.
\begin{lemma}\label{Lemma4-2}
For any $\bv= \{\bv_0, \bv_b \} \in \V_h$ and $w\in H^1(T)$ on the
element $T\in \T_h$, one has
\begin{eqnarray}
(\nabla_w\cdot \bv,\; \bbQ_hw)_T &=& -(\bv_0,\; \nabla w)_T +
\langle
(\bv_0-\bv_b)\cdot \bn,\; w-\bbQ_hw\rangle_{\partial T}\nonumber \\
& &  + \langle \bv_b\cdot\bn,\; w\rangle_{\partial
T}.\label{key1.before}
\end{eqnarray}
Moreover, if $w\in H^1(\Omega)$, then the sum of (\ref{key1.before})
over all $T\in \T_h$ gives
\begin{eqnarray}
b(\bv,\; \bbQ_hw) &=& -(\bv_0,\; \nabla_h w) +
\sum_{T\in\T_h}\langle
(\bv_0-\bv_b)\cdot \bn,\; w-\bbQ_hw\rangle_{\partial T}\nonumber \\
& &  + \langle \bv_b\cdot\bn,\; w\rangle_{\partial
\Omega},\label{key1}
\end{eqnarray}
where $\nabla_h w$ is the gradient of $w$ taken element-by-element.
\end{lemma}

\medskip

Let $T\in \T_h$ be any element, and $\bq\in [H^1(T)]^d$. Then, we
have from the definition of $Q_b$ that
\begin{eqnarray*}
\|\bq\cdot\bn - Q_b(\bq\cdot\bn)\|_{\partial T}^2 &=& \langle
\bq\cdot\bn - Q_b(\bq\cdot\bn), \ \bq\cdot\bn - Q_b(\bq\cdot\bn)
\rangle_{\partial T}\\
&=& \langle \bq\cdot\bn - Q_b(\bq\cdot\bn),\ \bq\cdot\bn -
(Q_0\bq)\cdot\bn)\rangle_{\partial T}\\
&\le&\|\bq\cdot\bn - Q_b(\bq\cdot\bn)\|_{\partial T}\ \|\bq\cdot\bn
- (Q_0\bq)\cdot\bn\|_{\partial T}.
\end{eqnarray*}
It follows that
\begin{equation}\label{fromb2zero}
\|\bq\cdot\bn - Q_b(\bq\cdot\bn)\|_{\partial T} \le \|\bq\cdot\bn -
(Q_0\bq)\cdot\bn\|_{\partial T}.
\end{equation}
\smallskip

\begin{lemma}\label{Lemma:5.3}
Let $\bq\in [H^{k+1}(\Omega)]^d$ and $s$ be any real number such that $0\le s\le
k$. Then, we have
\begin{equation}\label{L2boundary}
\sum_{T\in\T_h} h_T\|\bq\cdot\bn - Q_b(\bq\cdot\bn)\|_{\partial T}^2
\lesssim \ h^{2(s+1)} \|\bq\|_{s+1}^2,
\end{equation}
and
\begin{equation}\label{apprx}
\3bar \bq-Q_h\bq\3bar\lesssim  \rho h^{s+1}\|\bq\|_{s+1}.
\end{equation}
\end{lemma}

\begin{proof}
Apply the trace inequality (\ref{trace}) (see Appendix) to the
right-hand side of (\ref{fromb2zero}) to obtain
\begin{equation}\label{boundary-by-element}
h_T\|\bq\cdot\bn - Q_b(\bq\cdot\bn)\|_{\partial T}^2 \lesssim \|\bq
- Q_0\bq\|_T^2 + h_T^2\|\nabla(\bq-Q_0\bq)\|_{T}^2.
\end{equation}
Thus, it follows from Lemma \ref{L2errorbound} that the estimate
(\ref{L2boundary}) holds true.

To derive (\ref{apprx}), we have from the definition of the bilinear
form $s(\cdot,\cdot)$ in (\ref{c}), the definition of $Q_b$, and the
inequality (\ref{fromb2zero}) that
\begin{eqnarray*}
s(\bq-Q_h\bq,\bq-Q_h\bq)&&=\rho\sum_{T\in\T_h}h_T\|(\bq-Q_0\bq)\cdot\bn-(\bq\cdot\bn-Q_b(\bq\cdot\bn))\|_{\partial T}^2\\
&&\lesssim \rho\sum_{T\in\T_h}h_T\left(\|(\bq-Q_0\bq)\cdot\bn\|_{\partial T}^2+\|\bq\cdot\bn-Q_b(\bq\cdot\bn)\|_{\partial T}^2\right)\\
&&\lesssim\rho\sum_{T\in\T_h}h_T \|(\bq-Q_0\bq)\cdot\bn\|_{\partial
T}^2.
\end{eqnarray*}
Now apply the trace inequality (\ref{trace}) to the last term of the
above inequality to obtain
\begin{eqnarray}\label{feb-5.18}
s(\bq-Q_h\bq,\bq-Q_h\bq) \lesssim \ \rho\sum_{T\in\T_h}\left(
\|\bq-Q_0\bq\|_{T}^2+h_T^2\|\nabla(\bq-Q_0\bq)\|^2_T\right).
\end{eqnarray}
Finally, it follows from the triple-bar norm definition
(\ref{3barnorm}) and the above inequality that
\begin{eqnarray*}
\3bar \bq-Q_h\bq\3bar^2 &&=a(\bq-Q_h\bq, \bq-Q_h\bq)+s( \bq-Q_h\bq, \bq-Q_h\bq)\\
&&=\|\alpha^{\frac12}(\bq-Q_0\bq)\|^2+ s( \bq-Q_h\bq, \bq-Q_h\bq),\\
\end{eqnarray*}
which, combined with Lemma \ref{L2errorbound}, gives rise to the
desired estimate (\ref{apprx}). This completes the proof.
\end{proof}

\section{Error Analysis}\label{section-erroranalysis}
The goal of this section is to derive some optimal order error
estimates for the WG-MFEM approximation $(\bq_h;u_h)$ obtained from
(\ref{wg1})-(\ref{wg2}). In finite element analysis, it is routine
to decompose the error into two components in which the first is the
difference between the exact solution and a properly defined
projection of the exact solution while the second is the difference
of the finite element solution with the same projection. In the
current application, we shall employ the following decomposition:
\begin{eqnarray*}
\bq-\bq_h &=& (\bq - Q_h\bq) + (Q_h\bq - \bq_h),\\
u-u_h &=&(u-\bbQ_{h}u) + (\bbQ_{h}u-u_h).
\end{eqnarray*}
For simplicity, we introduce the following notation:
\[
\be_h=\{\be_0,\be_b\}:=\bq_h- Q_h\bq,\qquad \epsilon_h:=u_h-\bbQ_hu.
\]
Moreover, we shall denote $\sum_{T\in {\cal T}_h}\|\nabla w\|_T^2$
by $\|\nabla_h w\|^2$ for any $w\in \prod_{T\in \T_h} H^1(T)$.

\subsection{Error equations} The goal here is to derive equations
that $\be_h$ and $\epsilon_h$ must satisfy. To this end, for any
finite element function $\bv=\{\bv_0,\bv_b\}\in \V_h$, we test the
first equation of (\ref{mix}) against $\bv_0$ to obtain
\[
(\alpha\bq,\bv_0)+(\nabla u,\bv_0)=0.
\]
Now applying the identity (\ref{key1}) to the term $(\nabla u,
\bv_0)$ yields
\[
a(\bq,\bv)-b(\bv,\bbQ_hu)= \sum_{T\in {\cal T}_h}\langle
(\bv_0-\bv_b)\cdot\bn,\bbQ_hu- u\rangle_{\partial T} + \langle
g,\bv_b\cdot\bn \rangle_{\partial \Omega},
\]
where the Dirichlet boundary condition (\ref{bc1}) was also used.
Recall that $a_s(\cdot,\cdot)=a(\cdot,\cdot)+s(\cdot,\cdot)$. Adding
and subtracting $a_s(Q_h\bq,\bv)$ in the equation above gives
\begin{eqnarray}
a_s(Q_h\bq,\bv) -
b(\bv,\bbQ_hu)&=&a(Q_h\bq-\bq,\;\bv)+s(Q_h\bq,\bv)+ \langle
g,\bv_b\cdot\bn \rangle_{\partial
\Omega}\nonumber\\
&&+\sum_{T\in {\cal T}_h}\langle (\bv_0-\bv_b)\cdot\bn,\bbQ_hu-
u\rangle_{\partial T}.\label{eq3-old}
\end{eqnarray}
It follows from the definition of the bilinear form $s(\cdot,\cdot)$
that for any $\bw\in [H^1(\Omega)]^d$, one has
\begin{equation}\label{c-p1}
s(\bw,\bv)=0,\quad\forall\bv\in \V.
\end{equation}
Thus, we have $s(Q_h\bq,\bv)=s(Q_h\bq-\bq,\bv)$ and
$$
a(Q_h\bq-\bq,\;\bv)+s(Q_h\bq,\bv) = a_s(Q_h\bq-\bq,\bv).
$$
Substituting the above into (\ref{eq3-old}) yields
\begin{eqnarray}
a_s(Q_h\bq,\bv) - b(\bv,\bbQ_hu)&=&a_s(Q_h\bq-\bq,\;\bv)+ \langle
g,\bv_b\cdot\bn \rangle_{\partial
\Omega}\nonumber\\
&&+\sum_{T\in {\cal T}_h}\langle (\bv_0-\bv_b)\cdot\bn,\bbQ_hu-
u\rangle_{\partial T}.\label{eq3}
\end{eqnarray}

Next, we test the second equation of (\ref{mix}) against any $w\in
\W_h$ to obtain
\[
(\nabla\cdot\bq,w)=(f,w).
\]
Substituting (\ref{key}) into the above equation yields
\begin{equation}\label{eq4}
b(Q_h\bq,\;w)=(f,\;w)-\sum_{T\in {\cal
T}_h}\langle\bq\cdot\bn-Q_b(\bq\cdot\bn),\;w\rangle_{\partial T}.
\end{equation}
Now, subtracting (\ref{eq3}) from (\ref{wg1}) gives
\begin{eqnarray}
a_s(\be_h,\bv)-b(\bv,\epsilon_h) =a_s(\bq-Q_h\bq,\;\bv)-\sum_{T\in
{\cal T}_h}\langle (\bv_0-\bv_b)\cdot\bn,\;\bbQ_hu-
u\rangle_{\partial T},\label{err1}
\end{eqnarray}
for all $v\in \V_h$. Similarly, subtracting (\ref{eq4}) from
(\ref{wg2}) yields
\begin{equation}\label{err2}
b(\be_h,w)=\sum_{T\in {\cal
T}_h}\langle\bq\cdot\bn-Q_b(\bq\cdot\bn),\;w\rangle_{\partial T},
\end{equation}
for all $w\in \W_h$.

The equations (\ref{err1}) and (\ref{err2}) constitute governing
rules for the error term $\be_h$ and $\epsilon_h$. These are called
error equations.

\subsection{ Boundedness and {\em inf-sup} condition}
For any interior edge or flat face $e\in {\cal E}_h$, let $T_1$ and
$T_2$ be two elements sharing $e$ in common. The jump of $w\in \W_h$
on $e$ is given by
\[
\ljump w\rjump_e= \left\{
\begin{array}{ll}
w|_{_{\partial T_1}}-w|_{_{\partial T_2}},\; & e\in {\cal
E}_h^0,\\
w,\;  & e\in
\partial\Omega.
\end{array}
\right.
\]
We introduce a norm in $\V$ as follows
\begin{equation}\label{3barnorm}
\3bar \bv\3bar^2:=a(\bv,\bv)+s(\bv,\bv),\qquad \bv\in \V.
\end{equation}
Furthermore, let $\|\cdot\|_{1,h}$ be a norm in the space
$\W_h+H^1(\Omega)$ defined as follows
\begin{equation}\label{discreteH1norm}
\|w\|_{1,h}^2:=\sum_{T\in {\cal T}_h}\|\nabla w\|_T^2+\sum_{e\in
{\cal E}_h}h_e^{-1}\|Q_b\ljump w\rjump\|_e^2.
\end{equation}
It is not hard to see that $\|\cdot \|_{1,h}$ defines a norm in
$\W_h$ since $\|w\|_{1,h}=0$ would lead to $w=const$ on each element
$T$ and $\ljump w\rjump=Q_b\ljump w\rjump=0$ on each edge or flat
face $e\in {\cal E}_h$. It follows that $w=0$ on each element $T$.
As to $\3bar \cdot \3bar$, assume that $\3bar\bv \3bar=0$ for some
$\bv\in \V_h$; i.e.,
$$
(\alpha\bv_0, \bv_0) + \rho\sum_{T\in {\cal
T}_h}h_T\langle(\bv_0-\bv_b)\cdot\bn,\;(\bv_0-\bv_b)\cdot\bn
\rangle_{\partial T}=0.
$$
This implies that $\bv_0=0$ on each element $T$ and
$(\bv_0-\bv_b)\cdot\bn=0$ on each edge or flat face $e\in {\cal
E}_h$. Thus, we have $\bv_b\cdot\bn=0$ on each $e\in {\cal E}_h$.
Recall that $\bv_b$ is a vector parallel to $\bn$ on each $e\in
{\cal E}_h$. Hence, we must have $\bv_b=0$ on each $e\in {\cal
E}_h$. The other properties for a norm can be easily verified for
$\|\cdot\|_{1,h}$ and $\3bar\cdot\3bar$.

It is not hard to see that the triple-bar norm is essentially a
discrete $L^2$ norm, and the norm $\|\cdot\|_{1,h}$ is an
$H^1$-equivalent norm in the corresponding space.

\begin{lemma}\label{lemma5} For any $\varphi\in \W_h$, there exists
at least one $\bv\in \V_h$ such that
\begin{eqnarray}
|b(\bv,\varphi)|&=&\|\varphi\|_{1,h}^2,\label{b-p1}\\
\3bar\bv\3bar &\lesssim&\|\varphi\|_{1,h}.\label{b-p2}
\end{eqnarray}
Therefore, the following {\em inf-sup} condition is satisfied
\begin{equation}\label{inf-sup}
\sup_{\bv\in \V_h, \ \3bar\bv\3bar \neq 0}{|b(\bv, \varphi)|\over
\3bar \bv\3bar} \ge \beta \|\varphi\|_{1,h}
\end{equation}
for some positive constant $\beta$ independent of the meshsize $h$.
\end{lemma}
\begin{proof}
It follows from the definition of $b(\cdot,\cdot)$ and the discrete
divergence that
\begin{eqnarray*}
b(\bv,\varphi)&=&(\nabla_w\cdot\bv,\;\varphi)\\
&=&\sum_{T\in {\cal T}_h}\left( \langle \bv_b\cdot\bn,
\varphi\rangle_{\partial T}-(\bv_0,\;\nabla\varphi)_T\right)\\
&=& \sum_{e\in {\cal E}_h}\langle \bv_b\cdot\bn_e,
\ljump\varphi\rjump\rangle_e -\sum_{T\in {\cal
T}_h}(\bv_0,\;\nabla\varphi)_T,
\end{eqnarray*}
provided that the jump $\ljump\varphi\rjump$ was taken consistently
with the direction set ${\cal D}_h$. Now by taking $\bv_0=
-\nabla\varphi$ on each element $T$ and $\bv_b =
h_e^{-1}Q_b(\ljump\varphi\rjump) \bn_e$ on each edge or flat face
$e$ we arrive at
\begin{eqnarray*}
b(\bv,\varphi)&=&\sum_{T\in {\cal
T}_h}(\nabla\varphi,\;\nabla\varphi)_T+\sum_{e\in {\cal
E}_h}h_e^{-1}\langle Q_b\ljump\varphi\rjump,
Q_b\ljump\varphi\rjump\rangle_e=\|\varphi\|_{1,h}^2,
\end{eqnarray*}
which verifies (\ref{b-p1}). To verify (\ref{b-p2}), we apply the
definition of the triple-bar norm to the above chosen $\bv$ to
obtain
\begin{eqnarray*}
\3bar\bv\3bar^2&&=(\alpha\nabla_h\varphi,\nabla_h\varphi)+
\rho\sum_{T\in {\cal T}_h}h_T\|\nabla\varphi\cdot\bn+h_e^{-1}Q_b(\ljump\varphi\rjump) \bn_e\cdot\bn\|_{\partial T}^2\\
&&\lesssim \|\nabla_h\varphi\|^2+\sum_{e\in {\cal
E}_h}h_e^{-1}\|Q_b\ljump\varphi\rjump\|_e^2=\|\varphi\|_{1,h}^2,
\end{eqnarray*}
which is what (\ref{b-p2}) states for.
\end{proof}

\begin{lemma}\label{lemma5-1} The bilinear form $b(\cdot,\cdot)$ is bounded in $\V_h\times \W_h$. In other words, there
exists a constant $C$ such that
\begin{equation}\label{b-boundedness}
|b(\bv,\varphi)|\leq C \3bar\bv\3bar \ \|\varphi\|_{1,h}.
\end{equation}
\end{lemma}

\begin{proof}
We note from the definition of weak divergence that
\begin{eqnarray*}
b(\bv, \varphi) &=& (\nabla_w\cdot\bv, \varphi) \\
&=& \sum_{T\in\T_h}\left(\langle\bv_b\cdot\bn, \varphi
\rangle_{\partial T}-(\bv_0, \nabla\varphi)_T \right)\\
&=&\sum_{e\in{\cal E}_h} \langle \bv_b\cdot\bn_e,\ljump
\varphi\rjump \rangle_e - \sum_{T\in\T_h}(\bv_0, \nabla\varphi)_T.
\end{eqnarray*}
Using Cauchy-Schwarz we obtain
\begin{eqnarray}
|b(\bv, \varphi)|&\leq& \sum_{e\in{\cal E}_h} \|\bv_b\cdot\bn_e\|_e
\ \|\ljump \varphi\rjump \|_e + \sum_{T\in\T_h} \|\bv_0\|_T \
\|\nabla\varphi\|_T\nonumber\\
&\leq& \left(\sum_{e\in{\cal E}_h} h_e\|\bv_b\cdot\bn_e\|_e^2
\right)^{\frac12} \left( \sum_{e\in{\cal E}_h} h_e^{-1} \|\ljump
\varphi\rjump \|_e^2\right)^{\frac12} + \|\bv_0\|\
\|\nabla_h\varphi\|\nonumber\\
&\leq& \left( \left(\sum_{e\in{\cal E}_h} h_e\|\bv_b\cdot\bn_e\|_e^2
\right)^{\frac12} +  \|\bv_0\| \right) \
\|\varphi\|_{1,h}.\label{tg.800}
\end{eqnarray}
Let $T$ be an element that takes $e$ as an edge or flat face. Then,
using the trace inequality (\ref{trace}) and the inverse inequality
(\ref{aaa.88-new}) we obtain
\begin{eqnarray*}
h_e \|\bv_b\cdot\bn_e\|_e^2 &\leq& 2
h_e\|(\bv_b-\bv_0)\cdot\bn_e\|_e^2 +
2h_e\|\bv_0\cdot\bn_e\|_e^2 \\
&\lesssim&h_e\|(\bv_b-\bv_0)\cdot\bn_e\|_e^2 + \|\bv_0\|_T^2.
\end{eqnarray*}
Substituting the above inequality into (\ref{tg.800}) yields
$$
|b(\bv, \varphi)| \lesssim \3bar\bv\3bar \ \|\varphi\|_{1,h},
$$
which completes the proof of the lemma.
\end{proof}

\medskip

\subsection{Error estimates} We first establish some estimates useful in the
forthcoming error estimates.

\begin{lemma}\label{lemma4} Let $u\in H^{k+2}(\Omega)$ and $\bq\in
[H^{k+1}(\Omega)]^d$ be two smooth functions on $\Omega$. Then, the
following estimates hold true
\begin{eqnarray}
\sum_{T\in {\cal T}_h}\left|\langle (\bv_0-\bv_b)\cdot\bn,\;\bbQ_hu-
u\rangle_{\partial T}\right|
\lesssim h^{s+1}\ \|u\|_{s+2}\3bar\bv\3bar,\;\; 0\le s\le k,\label{b1}\\
\sum_{T\in {\cal T}_h}\left|\langle\bq\cdot\bn
-Q_b(\bq\cdot\bn),\varphi\rangle_{\partial T}\right| \lesssim
h^{s+1}\ \|\bq\|_{s+1}\|\nabla_h\varphi\|,\;\;0\le s\le k,\label{b2}
\end{eqnarray}
for all $\bv\in \V$ and $\varphi\in \prod_{T\in\T_h} H^1(T)$.
\end{lemma}

\begin{proof}
It follows from the Cauchy-Schwarz inequality, the definition of
$\3bar\cdot\3bar$, and the trace inequality (\ref{trace}) that
\begin{eqnarray*}
\sum_{T\in {\cal T}_h}\left|\langle (\bv_0-\bv_b)\cdot\bn,\;\bbQ_hu-
u\rangle_{\partial T}\right| &&\lesssim \left(\sum_{T\in {\cal
T}_h}h_T\|(\bv_0-\bv_b)\cdot\bn\|_{\partial T}^2\right)^\frac12 \\
&& \quad \cdot \left(\sum_{T\in {\cal T}_h}h_T^{-1}\|\bbQ_hu-u\|_{\partial T}^2\right)^\frac12\\
&&\lesssim \3bar\bv\3bar\left(\sum_{T\in {\cal T}_h} \left(h_T^{-2}\|\bbQ_hu-u\|_T^2+\|\nabla (\bbQ_hu-u)\|^2_T\right)\right)^{\frac12}\\
&&\lesssim h^{s+1}\3bar\bv\3bar\ \|u\|_{s+2},
\end{eqnarray*}
where we have used the estimate (\ref{Jthree-term}). This verifies
the validity of (\ref{b1}).

Next, let $\bar{\varphi}$ be the average of $\varphi$ over each
element $T$. It follows from the definition of $Q_b$, the estimates
(\ref{L2boundary}), and (\ref{trace}) that
\begin{eqnarray*}
&&\sum_{T\in {\cal
T}_h}\left|\langle\bq\cdot\bn-Q_b(\bq\cdot\bn),\varphi\rangle_{\partial
T}\right| =\sum_{T\in {\cal
T}_h}\left|\langle\bq\cdot\bn-Q_b(\bq\cdot\bn),\varphi-\bar{\varphi}
\rangle_{\partial T}\right|\\
&&\le \left(\sum_{T\in\T_h}
h_T\|\bq\cdot\bn-Q_b(\bq\cdot\bn)\|_{\partial T}^2\right)^{\frac12}
\left(\sum_{T\in {\cal T}_h}h_T^{-1}\|\varphi-\bar{\varphi}\|^2_{\partial T}\right)^{\frac12}\\
&&\lesssim h^{s+1}\|\bq\|_{s+1}\left(\sum_{T\in {\cal
T}_h}\left(h_T^{-2}\|\varphi-\bar{\varphi}\|^2_T+
\|\nabla \varphi\|^2_T\right)\right)^{\frac12}\\
&&\lesssim h^{s+1}\|\bq\|_{s+1}\|\nabla_h\varphi\|,
\end{eqnarray*}
which completes the proof of (\ref{b2}).
\end{proof}

\smallskip

With the help of Lemma \ref{lemma4}, Lemma \ref{lemma5}, and Lemma
\ref{lemma5-1}, we are now in a position to derive an error estimate
for the WG-MFEM solution $(\bq_h;u_h)$ in the norm $\3bar\cdot\3bar$
and $\|\cdot\|_{1,h}$ respectively.

\begin{theorem}\label{thm1}
Let $\bq_h\in \V_h$ and $u_h\in \W_h$ be the solution of the weak
Galerkin mixed finite element scheme (\ref{wg1})-(\ref{wg2}). Assume
that the exact solution $(\bq; u)$ of (\ref{w-mix1})-(\ref{w-mix2})
is regular such that $u\in H^{s+2}(\Omega)$ and $\bq\in
[H^{s+1}(\Omega)]^d$ with $s\in (0,k]$. Then, one has
\begin{eqnarray}
\3bar \bq_h-Q_h\bq\3bar+\|u_h-\bbQ_hu\|_{1,h}&&\lesssim
h^{s+1}\left(\|u\|_{s+2}+\|\bq\|_{s+1}\right).\label{qh-qq}
\end{eqnarray}
\end{theorem}
\begin{proof} Let
$$
\ell(\bv):=a_s(\bq-Q_h\bq,\;\bv)-\sum_{T\in {\cal T}_h}\langle
(\bv_0-\bv_b)\cdot\bn,\;\bbQ_hu- u\rangle_{\partial T}
$$
be a linear functional on the finite element space $\V_h$. Also, let
$$
\chi(w):=\sum_{T\in {\cal
T}_h}\langle\bq\cdot\bn-Q_b(\bq\cdot\bn),\;w\rangle_{\partial T}
$$
be a linear functional on the finite element space $\W_h$. The error
equations (\ref{err1}) and (\ref{err2}) indicate that the pair
$(\be_h;\ \epsilon_h)$ is a solution of the following problem
\begin{eqnarray}
a_s(\be_h,\bv)-b(\bv,\epsilon_h)
&=&\ell(\bv),\qquad \forall\bv\in \V_h,\label{error-eq.01}\\
b(\be_h,w)&=&\chi(w),\qquad \forall w\in \W_h.\label{error-eq.02}
\end{eqnarray}

The bilinear form $a_s(\cdot,\cdot)$ is clearly bounded, symmetric
and positive definite in $\V_h$ equipped with the triple-bar norm
$\3bar\cdot\3bar$. Lemma \ref{lemma5-1} indicates that the bilinear
form $b(\cdot,\cdot)$ is bounded in $\V_h\times \W_h$, and Lemma
\ref{lemma5} proved that the usual {\em inf-sup} condition of
Bab\u{u}ska \cite{babuska} and Brezzi \cite{brezzi} is satisfied.
Thus, we have from the general theory of Bab\u{u}ska and Brezzi that
\begin{equation}\label{error-estimate-main}
\3bar \be_h\3bar + \|\epsilon_h\|_{1,h} \leq C\left(\|\ell\|_{\V'_h}
+ \|\chi\|_{\W'_h}\right),
\end{equation}
where $\|\ell\|_{\V'_h}$ and $\|\chi\|_{\W'_h}$ are the norm of the
corresponding linear functionals. To estimate the norms, we use
(\ref{apprx}) and (\ref{b1}) to come up with
$$
|\ell(\bv)| \lesssim h^{s+1}\left(\|\bq\|_{s+1} +
\|u\|_{s+2}\right)\3bar\bv\3bar.
$$
Thus, we have
\begin{equation}\label{nov27.01}
\|\ell\|_{\V'_h}\lesssim h^{s+1}\left(\|\bq\|_{s+1} +
\|u\|_{s+2}\right),\qquad 0\le s \le k.
\end{equation}
As to the norm of $\chi$, we use the estimate (\ref{b2}) to obtain
$$
|\chi(w)|\lesssim h^{s+1} \|\bq\|_{s+1}\|w\|_{1,h},\qquad 0\le s \le
k,
$$
which implies that
\begin{equation}\label{nov27.02}
\|\chi\|_{\W'_h}\lesssim h^{s+1} \|\bq\|_{s+1},\qquad 0\le s \le k.
\end{equation}
Substituting (\ref{nov27.01}) and (\ref{nov27.02}) into
(\ref{error-estimate-main}) yields the desired error estimate
(\ref{qh-qq}).
\end{proof}

\section{Error Estimates in $L^2$}\label{error-analysis-ell2}
To obtain an optimal order error estimate for the scalar component
$\epsilon_h=u_h-\bbQ_hu$ in the usual $L^2$-norm, we consider a dual
problem that seeks $\Psi$ and $\phi$ satisfying
\begin{eqnarray}
\alpha\Psi+\nabla \phi&=&0,\; \quad
\mbox{in}\;\Omega\label{dual1}\\
\nabla\cdot \Psi &=&\epsilon_h,\quad
\mbox{in}\;\Omega\label{dual2}\\
\phi&=&0,\; \quad \mbox{on}\; \partial\Omega.\label{dual-bc}
\end{eqnarray}
Assume that the usual $H^2$-regularity is satisfied for the dual
problem; i.e., for any $\epsilon_h\in L^2(\Omega)$, there exists a
unique solution $(\Psi;\phi) \in [H^1(\Omega)]^d\times H^2(\Omega)$
such that
\begin{equation}\label{reg}
\|\phi\|_2+\|\Psi\|_1\lesssim \|\epsilon_h\|.
\end{equation}

\smallskip

\begin{lemma}
Let $(\bq; u)$ be the solution of (\ref{w-mix1})-(\ref{w-mix2}).
Assume that $u\in H^{k+2}(\Omega)$, $\bq\in [H^{k+1}(\Omega)]^d$,
and the coefficient function satisfies $\alpha\in W^{1,\infty}(T)$
on each element $T$. Then, one has the following estimates
\begin{eqnarray}
|a_s(\bq-Q_h\bq, Q_h\Psi)|&&\lesssim h^{k+2}\|\bq\|_{k+1}\|\Psi\|_1,\label{l1}\\
\sum_{T\in {\cal T}_h} \left|\langle Q_0\Psi\cdot\bn-Q_b(\Psi\cdot
\bn), \bbQ_hu-u\rangle_{\partial T}\right|&& \lesssim h^{k+2}
\|u\|_{k+2}\|\Psi\|_1\label{l2},\\
\sum_{T\in {\cal T}_h} \left|\langle \bq\cdot\bn - Q_b(\bq\cdot\bn),
\bbQ_h\phi-\phi\rangle_{\partial T}\right|&&\lesssim
h^{k+2}\|\bq\|_{k+1}\|\phi\|_2\label{l9}.
\end{eqnarray}
\end{lemma}
\begin{proof} Let $\bar\alpha$ be the average of $\alpha$ on each
element $T$. It follows from the definition of $a_s(\cdot,\cdot)$
and the approximation property (\ref{apprx}) that
\begin{eqnarray*}
|a_s(\bq-Q_h\bq,\; Q_h\Psi)|&&\le|(\alpha(\bq-Q_0\bq),\; Q_0\Psi)|+|s(\bq-Q_0\bq, Q_h\Psi)|\\
&&=|(\bq-Q_0\bq,\; (\alpha-\bar\alpha)Q_0\Psi)|+ |s(\bq-Q_0\bq, Q_h\Psi-\Psi)|\\
&&\le Ch \|\nabla\alpha\|_{\infty} \|\bq-Q_0\bq\|\ \|Q_0\Psi\|+\3bar\bq-Q_h\bq\3bar\ \3bar\Psi-Q_h\Psi\3bar\\
&&\lesssim h^{k+2}\|\bq\|_{k+1}\|\Psi\|_1,
\end{eqnarray*}
where $\|\nabla\alpha\|_{\infty}$ is the $L^\infty$ norm of
$\nabla\alpha$ taken on each element. As to (\ref{l2}), we use
(\ref{trace}) and (\ref{apprx}) to obtain
\begin{eqnarray*}
&&\sum_{T\in {\cal T}_h} \left|\langle  Q_0\Psi\cdot \bn -Q_b(\Psi\cdot\bn), \bbQ_hu-u\rangle_{\partial T}\right|\\
&&\quad =\sum_{T\in {\cal T}_h} \left|\langle  (Q_0\Psi -\Psi)\cdot
\bn-(Q_b(\Psi\cdot\bn)-\Psi\cdot\bn),
\bbQ_hu-u\rangle_{\partial T}\right|\\
&&\quad \lesssim \3bar Q_h\Psi-\Psi\3bar \left(\sum_{T\in {\cal T}_h}(h_T^{-2}\|\bbQ_hu-u\|_T^2
+\|\nabla(\bbQ_hu-u)\|_T^2)\right)^{1/2}\\
&&\quad \lesssim h^{k+2}\|u\|_{k+2}\|\Psi\|_1.
\end{eqnarray*}
Similarly, using (\ref{L2boundary}) we have
\begin{eqnarray*}
&&\sum_{T\in {\cal T}_h}\left|\langle\bq\cdot\bn-Q_b(\bq\cdot\bn),\;
\bbQ_h\phi-\phi\rangle_{\partial T}\right|\\
&&\quad \lesssim h^{k+1}\|\bq\|_{k+1}\left(\sum_{T\in {\cal
T}_h}(h_T^{-2}\|\bbQ_h\phi-\phi\|^2_{T}
+\|\nabla (\bbQ_h\phi-\phi)\|^2_{T})\right)^{\frac12}\\
&&\quad \lesssim h^{k+2}\|\bq\|_{k+1}\|\phi\|_2.
\end{eqnarray*}
This completes the proof of the lemma.
\end{proof}

\medskip

We are now in a position to establish an optimal order error
estimate for the scalar/pressure component $\epsilon_h=u_h-\bbQ_hu$
in the usual $L^2$-norm. The result is stated in the following
theorem.

\smallskip

\begin{theorem}\label{thm2}
Let $\bq_h\in \V_h$ and $u_h\in \W_h$ be the solution of
(\ref{wg1})-(\ref{wg2}). Assume that the exact solution of
(\ref{w-mix1})-(\ref{w-mix2}) satisfies $u\in H^{k+2}(\Omega)$ and
$\bq\in [H^{k+1}(\Omega)]^d$. Then, one has the following error
estimate
\begin{equation}\label{L2errorestimate}
\|u_h-\bbQ_hu\|\lesssim
h^{k+2}\left(\|u\|_{k+2}+\|\bq\|_{k+1}\right).
\end{equation}
\end{theorem}
\begin{proof}
Testing (\ref{dual2}) by $\epsilon_h=u_h-\bbQ_hu$ and then using
(\ref{key}) yields
\begin{eqnarray}
(\epsilon_h,\;\epsilon_h)&=&(\nabla\cdot\Psi,\;\epsilon_h)\nonumber
\\&=& b(Q_h\Psi,\; \epsilon_h) +\sum_{T\in {\cal T}_h} \langle
\Psi\cdot\bn-Q_b(\Psi\cdot\bn),\; \epsilon_h\rangle_{\partial
T}.\label{jpw-05}
\end{eqnarray}
It follows from (\ref{err1}) that
\begin{eqnarray*}
b(Q_h\Psi,\; \epsilon_h)&&= a_s(\be_h,\; Q_h\Psi)-a_s(\bq-Q_h\bq,\; Q_h\Psi)\\
&& +\sum_{T} \langle (Q_0\Psi)\cdot \bn-Q_b(\Psi\cdot\bn),\;
\bbQ_hu-u\rangle_{\partial T}.
\end{eqnarray*}
Substituting the above equation into (\ref{jpw-05}) yields
\begin{eqnarray}
(\epsilon_h,\;\epsilon_h)&=& a_s(\be_h,\; Q_h\Psi) -a_s(\bq-Q_h\bq,\; Q_h\Psi)\nonumber\\
 & & +\sum_{T\in\T_h} \langle Q_0\Psi\cdot \bn-Q_b(\Psi\cdot\bn),\; \bbQ_hu-u\rangle_{\partial
 T}\label{hello.08}\\
 & & +\sum_{T\in {\cal T}_h} \langle \Psi\cdot\bn-Q_b(\Psi\cdot\bn),\; \epsilon_h
 \rangle_{\partial T}.\nonumber
\end{eqnarray}
Using (\ref{l2}) we obtain
\begin{equation}\label{hi.188}
\sum_{T\in\T_h} \langle Q_0\Psi\cdot \bn-Q_b(\Psi\cdot\bn),\;
\bbQ_hu-u\rangle_{\partial T} \lesssim h^{k+2} \|u\|_{k+2} \
\|\Psi\|_1.
\end{equation}
Using (\ref{b2}) we arrive at
\begin{equation}\label{hi.189}
\sum_{T\in {\cal T}_h} \langle \Psi\cdot\bn-Q_b(\Psi\cdot\bn),\;
\epsilon_h\rangle_{\partial T} \lesssim h\|\nabla_h\epsilon_h\|\
\|\Psi\|_1.
\end{equation}
From (\ref{l1}) we have
\begin{equation}\label{hi.190}
|a_s(\bq-Q_h\bq,\; Q_h\Psi)|\lesssim h^{k+2}\|\bq\|_{k+1}\
\|\Psi\|_1.
\end{equation}
Now substituting (\ref{hi.188})-(\ref{hi.190}) into (\ref{hello.08})
we obtain
\begin{equation}\label{hi.191}
(\epsilon_h,\;\epsilon_h)\lesssim |a_s(\be_h,\; Q_h\Psi)| +
h^{k+2}\left(\|\bq\|_{k+1} +\|u\|_{k+2}\right)\|\Psi\|_1 +
h\|\nabla_h \epsilon_h\|\ \|\Psi\|_1.
\end{equation}

It remains to deal with $|a_s(\be_h, Q_h\Psi)|$ in (\ref{hi.191}).
To this end, we note that
$$
a_s(\be_h, Q_h\Psi) = a_s(\be_h, Q_h\Psi-\Psi) + a_s(\be_h, \Psi),
$$
and
$$
|a_s(\be_h, Q_h\Psi-\Psi)|\le \3bar\be_h\3bar \ \3bar
Q_h\Psi-\Psi\3bar \lesssim h\3bar\be_h\3bar \ \|\Psi\|_1.
$$
Thus,
\begin{equation}\label{hi.201}
|a_s(\be_h, Q_h\Psi)| \lesssim |a_s(\be_h, \Psi)| + h\3bar\be_h\3bar
\ \|\Psi\|_1.
\end{equation}
Since $\Psi\in [H^1(\Omega)]^d$, then we have $a_s(\be_h,
\Psi)=a(\be_h, \Psi)$; i.e., the stabilization term vanishes when
one of the components is sufficiently regular. Now testing
(\ref{dual1}) against $\be_h$ gives
$$
a(\be_h, \Psi) = (\alpha\Psi, \be_0) = - (\nabla\phi, \be_0).
$$
Furthermore, we have from (\ref{key1}) (with $\bv=\be_h$ and
$w=\phi$) and the fact that $\phi=0$ on $\partial\Omega$ that
$$
(\nabla\phi, \be_0) = -b(\be_h, \bbQ_h\phi)
 + \sum_{T\in {\cal T}_h} \langle (\be_0-\be_b)\cdot\bn, \phi-\bbQ_h\phi\rangle_{\partial T}.
$$
Using the error equation (\ref{err2}) to replace the term $b(\be_h,
\bbQ_h\phi)$ in above equation we obtain
\begin{eqnarray*}
(\nabla\phi, \be_0) &=& -\sum_{T\in {\cal T}_h} \langle \bq\cdot\bn
- Q_b(\bq\cdot\bn), \bbQ_h\phi\rangle_{\partial T}
 + \sum_{T\in {\cal T}_h} \langle (\be_0-\be_b)\cdot\bn, \phi-\bbQ_h\phi\rangle_{\partial
 T}\\
&=& \sum_{T\in {\cal T}_h} \langle \bq\cdot\bn - Q_b(\bq\cdot\bn),
\phi-\bbQ_h\phi\rangle_{\partial T}
 + \sum_{T\in {\cal T}_h} \langle (\be_0-\be_b)\cdot\bn, \phi-\bbQ_h\phi\rangle_{\partial
 T},
\end{eqnarray*}
where we have used the fact that $\phi\in H^1(\Omega)$ and $\phi=0$
on $\partial\Omega$ in the second line. Now using (\ref{l9}) to
estimate the first summation in the above equation, and (\ref{b1})
with $s=0$, $\bv=\be_h$, $u=\phi$ to estimate the second summation
we obtain
\begin{eqnarray*}
|(\nabla\phi, \be_0)|
 \lesssim (h^{k+2}\|\bq\|_{k+1} + h \3bar\be_h\3bar) \ \|\phi\|_2.
\end{eqnarray*}
Thus,
$$
|a_s(\be_h,\Psi)|=|a(\be_h,\Psi)|=|(\nabla\phi, \be_0)|\lesssim
(h^{k+2}\|\bq\|_{k+1} + h \3bar\be_h\3bar) \ \|\phi\|_2.
$$
Substituting the above into (\ref{hi.201}) yields
\begin{equation}\label{hi.308}
|a_s(\be_h,Q_h\Psi)|\lesssim(h^{k+2}\|\bq\|_{k+1} + h
\3bar\be_h\3bar) \ \|\phi\|_2 + h \3bar\be_h\3bar \ \|\Psi\|_1.
\end{equation}
Now substituting (\ref{hi.308}) into (\ref{hi.191}) gives
\begin{eqnarray}\label{hi.891}
(\epsilon_h,\;\epsilon_h)\lesssim && (h^{k+2}\|\bq\|_{k+1} + h
\3bar\be_h\3bar) \ \|\phi\|_2 + h \3bar\be_h\3bar \ \|\Psi\|_1 \\
&& + h^{k+2}\left(\|\bq\|_{k+1} +\|u\|_{k+2}\right)\|\Psi\|_1 +
h\|\nabla_h \epsilon_h\|\ \|\Psi\|_1.\nonumber
\end{eqnarray}

Observe that the $H^2$-regularity of the dual problem implies the
following estimate
$$
\|\phi\|_2 + \|\Psi\|_1\lesssim \|\epsilon_h\|.
$$
Substituting the above into (\ref{hi.891}) and then dividing both
sides by $\|\epsilon_h\|$ we obtain
$$
\|\epsilon_h\|\lesssim  h (\3bar\be_h\3bar+\|\nabla_h \epsilon_h\|)
+ h^{k+2}\left(\|\bq\|_{k+1} +\|u\|_{k+2}\right),
$$
which, together with (\ref{qh-qq}), implies the desired $L^2$ error
estimate (\ref{L2errorestimate}). This completes the proof.
\end{proof}

\appendix

\section*{Appendix}
\setcounter{section}{1} The goal of this Appendix is to establish
some fundamental estimates useful in the error estimate for WG-MFEM.
First, we derive a trace inequality for functions defined on the
finite element partition $\T_h$ with properties specified in Section
\ref{wg-mfem}.

\begin{lemma}[Trace Inequality] Let $\T_h$ be a partition of the domain $\Omega$ into
polygons in 2D or polyhedra in 3D. Assume that the partition $\T_h$
satisfies the assumptions (A1), (A2), and (A3) as specified in
Section \ref{wg-mfem}. Then, there exists a constant $C$ such that
for any $T\in\T_h$ and edge/face $e\in\partial T$, we have
\begin{equation}\label{trace}
\|\theta\|_{e}^2 \leq C \left( h_T^{-1} \|\theta\|_T^2 + h_T
\|\nabla \theta\|_{T}^2\right),
\end{equation}
where $\theta\in H^1(T)$ is any function.
\end{lemma}
\begin{proof}
We shall provide a proof for the case of $e$ being a flat face and
$\theta\in C^1(T)$. To this end, let the flat face $e$ be
represented by the following parametric equation:
\begin{eqnarray}\label{parametric-e}
\bx_e=\phi(\xi,\eta):=(\phi_1(\xi, \eta),
\phi_2(\xi,\eta),\phi_3(\xi,\eta))
\end{eqnarray}
for $(\xi,\eta)\in D\subset \mathbb{R}^2$. By Assumption {\bf A3},
there exists a pyramid $P(e,T,A_e)$ with apex
$A_e=\bx_*:=(x_1^*,x_2^*,x_3^*)$ contained in the element $T$, see
Figure \ref{fig:shape-regular-element}. This pyramid has the
following parametric representation:
\begin{eqnarray}\label{pyramid}
\bx(t,\xi,\eta)&=&(1-t)\phi(\xi, \eta)+t\bx_*
\end{eqnarray}
for $(t, \xi,\eta)\in [0,1]\times D$. For any given $\bx_e\in e$,
the line segment joining $\bx_e$ and the apex $\bx_*$ can be
represented by
$$
\bx(t) = \bx_e + t(\bx_*-\bx_e).
$$
From the fundamental theorem of Calculus,  we have
$$
\theta^2(\bx_e)-\theta^2(\bx(t))=-\int_0^{t}\partial_\tau\theta^2(\bx_e+\tau\omega)d\tau,
\qquad\omega = {\bx_*-\bx_e}.
$$
The above can be further rewritten as
$$
\theta^2(\bx_e)-\theta^2(\bx(t))=-2\int_0^{t}\theta \left(
\nabla\theta\cdot \omega\right) d\tau.
$$
It follows from the Cauchy-Schwarz inequality that for any $t\in
[0,\frac12]$ we have
$$
\theta^2(\bx_e) \le \theta^2(\bx(t)) +
h_T^{-1}\int_0^{\frac12}\theta^2 |\omega| d\tau +h_T
\int_0^{\frac12}|\nabla\theta|^2 |\omega| d\tau.
$$
Now we integrate the above inequality over the flat face $e$ by
using the parametric equation (\ref{parametric-e}), yielding
\begin{eqnarray}
\int_D\theta^2(\bx_e)|\phi_\xi\times\phi_\eta| d\xi d\eta &\le&
\int_D\theta^2(\bx(t))|\phi_\xi\times\phi_\eta| d\xi d\eta \nonumber\\
& & +h_T^{-1}\int_0^{\frac12}\int_D\theta^2 |\omega|
|\phi_\xi\times\phi_\eta|
d\xi d\eta d\tau \label{aaa.01}\\
& & +h_T\int_0^{\frac12}\int_D|\nabla\theta|^2 |\omega|
|\phi_\xi\times\phi_\eta| d\xi d\eta d\tau.\label{aaa.02}
\end{eqnarray}
Observe that the integral of a function over the following
prismatoid
$$
P_{\frac12}:=\left\{\bx(t,\xi,\eta):\quad (t,\xi,\eta)\in
[0,1/2]\times D \right\}
$$
is given by
$$
\int_{P_{\frac12}} f(\bx) d\bx= \int_{0}^{\frac12}\int_D
f(\bx(\tau,\xi,\eta)) J(\tau,\xi,\eta) d\xi d\eta d\tau,
$$
where $J(\tau,\xi,\eta)=(1-\tau)^2 |(\phi_\xi\times \phi_\eta)\cdot
\omega|$ is the Jacobian from the coordinate change. The vector
$\phi_\xi\times \phi_\eta$ is normal to the face $e$, and
$\omega=\bx_*-\bx_e$ is a vector from the base point $\bx_e$ to the
apex $\bx_*$. The angle assumption (see Assumption {\bf A3} of
Section \ref{wg-mfem}) for the prism $P(e,T,A_e)$
indicates that the Jacobian satisfies the following relation
\begin{equation}\label{jacobian-estimate}
J(\tau,\xi,\eta)\ge \frac{\mu_0}{4} |\phi_\xi\times \phi_\eta|\
|\omega|,\quad \tau\in [0,1/2]
\end{equation}
for some fixed $\mu_0\in (0,1)$. Observe that the
left-hand side of (\ref{aaa.01}) is the surface integral of
$\theta^2$ over the face $e$. Thus, substituting
(\ref{jacobian-estimate}) into (\ref{aaa.01}) and (\ref{aaa.02})
yields
\begin{eqnarray*}
\int_e\theta^2 de &\le&
\int_D\theta^2(\bx(t))|\phi_\xi\times\phi_\eta| d\xi d\eta \nonumber\\
& & +4\mu_0^{-1}h_T^{-1}\int_{P_{\frac12}} \theta^2 d\bx
+4\mu_0^{-1}h_T\int_{P_{\frac12}}|\nabla\theta|^2 d\bx.
\end{eqnarray*}
Now we integrate the above with respect to $t$ in the interval
$[0,\frac12]$ to obtain
\begin{eqnarray}
\frac12 \int_e\theta^2 de &\le&
\int_0^{\frac12}\int_D\theta^2(\bx(t))|\phi_\xi\times\phi_\eta| d\xi d\eta dt\label{aaa.03}\\
& & +2\mu_0^{-1}h_T^{-1}\int_{P_{\frac12}} \theta^2 d\bx
+2\mu_0^{-1}h_T\int_{P_{\frac12}}|\nabla\theta|^2 d\bx.\nonumber
\end{eqnarray}
Again, by substituting (\ref{jacobian-estimate}) into the right-hand
side of (\ref{aaa.03}) we arrive at
\begin{eqnarray}
\frac12 \int_e\theta^2 de &\le& 4\mu_0^{-1} |\omega|^{-1}
\int_0^{\frac12}\int_D\theta^2(\bx(t))J(t,\xi,\eta) d\xi d\eta dt
\label{aaa.04}\\
& & +2\mu_0^{-1}h_T^{-1}\int_{P_{\frac12}} \theta^2 d\bx
+2\mu_0^{-1}h_T\int_{P_{\frac12}}|\nabla\theta|^2 d\bx.\nonumber
\end{eqnarray}
The integral on the right-hand side of (\ref{aaa.04}) is the
integral of $\theta^2$ on the prismatoid $P_{\frac12}$. It can be
seen from the Assumption {\bf A3} that
$$
|\omega|^{-1} \leq \alpha_* h_T^{-1}
$$
for some positive constant $\alpha_*$. Therefore, we have from the
above estimate that
$$
\int_e \theta^2 de \leq C h_T^{-1} \int_{P_{\frac12}} \theta^2 d\bx
+ C h_T\int_{P_{\frac12}} |\nabla\theta|^2 d\bx,
$$
which completes the proof of the Lemma.
\end{proof}
\medskip

Next, we would like to establish an estimate for the $L^2$ norm of polynomial functions by their $L^2$ norm on a subdomain. To this end, we first derive a result of similar nature for general functions in $H^1$.

\begin{lemma}\label{appendix-lemma2}
Let $K\subset\mathbb{R}^d$ be convex and $v\in H^1(K)$. Then,
\begin{equation}\label{aaa.08}
\|v\|^2_{K} \leq \frac{2|K|}{|S|} \|v\|_{S}^2
+\frac{4\delta^{2(d+1)}}{|S|^2d^2}\|\nabla v\|_K^2,
\end{equation}
where $\delta$ is the diameter of $K$ and $S$ is any measurable
subset of $K$.
\end{lemma}

\begin{proof}
Since $C^1(K)$ is dense in $H^1(K)$, it is sufficient to establish
(\ref{aaa.08}) for $v\in C^1(K)$. For any $\bx, \by\in K$, we have
$$
v^2(\bx)=v^2(\by) - \int_0^{|\bx-\by|} \partial_r v^2(\bx+r\omega)
dr,\quad \omega=\frac{\by-\bx}{|\by-\bx|}.
$$
From the usual chain rule and the Cauchy-Schwarz inequality we obtain
\begin{eqnarray}
v^2(\bx)&=&v^2(\by) - 2\int_0^{|\bx-\by|} v \partial_r v(\bx+r\omega)
dr\nonumber\\
&\le & v^2(\by) + \epsilon\delta^{-1} \int_0^{|\bx-\by|} v^2 dr + \epsilon^{-1}\delta\int_0^{|\bx-\by|} |\partial_r v|^2 dr,\label{aaa.10}
\end{eqnarray}
where $\epsilon>0$ is any constant. Let
$$
V(\bx)=\left\{
\begin{array}{ll}
v^2(\bx),&\qquad\bx\in K\\
0,&\qquad\bx\notin K
\end{array}
\right.
$$
and
$$
W(\bx)=\left\{
\begin{array}{ll}
|\partial_r v|^2(\bx),&\qquad\bx\in K\\
0,&\qquad\bx\notin K.
\end{array}
\right.
$$
Then, the inequality (\ref{aaa.10}) can be rewritten as
$$
v^2(\bx)\le v^2(\by) + \epsilon\delta^{-1} \int_0^{\infty} V(\bx+r\omega)dr + \epsilon^{-1}\delta\int_0^{\infty}
W(\bx+r\omega)dr.
$$
Integrating the above inequality with respect to $\by$ in $S$ yields
\begin{eqnarray}
& & |S| v^2(\bx) \leq \int_S v^2 dS \label{aaa.16}\\
& & \ + \int_{|\bx-\by|\le \delta}\left(\epsilon\delta^{-1}\int_0^{\infty} V(\bx+r\omega)dr+\epsilon^{-1}\delta \int_0^{\infty} W(\bx+r\omega)dr  \right)d\by.\nonumber
\end{eqnarray}
It is not hard to see that
\begin{eqnarray}
\int_{|\bx-\by|\le \delta}\int_0^{\infty} V(\bx+r\omega)dr d\by &=& \int_0^\infty\int_{|\omega|=1}\int_0^\delta V(\bx+r\omega) \rho^{d-1}\ d\rho\ d\omega\ dr\nonumber\\
&=&\frac{\delta^d}{d}\int_0^\infty\int_{|\omega|=1}V(\bx+r\omega) d\omega\ dr\nonumber\\
&=&\frac{\delta^d}{d}\int_K |\bx-\by|^{1-d} v^2(\by)d\by.\label{aaa.17}
\end{eqnarray}
Analogously, we have
\begin{equation}\label{aaa.18}
\int_{|\bx-\by|\le \delta}\int_0^{\infty} W(\bx+r\omega)dr d\by =\frac{\delta^d}{d}\int_K |\bx-\by|^{1-d} |\partial_r v(\by)|^2d\by.
\end{equation}
Substituting (\ref{aaa.17}) and (\ref{aaa.18}) into (\ref{aaa.16}) yields
$$
|S| v^2(\bx) \leq \int_S v^2 dS + \frac{\epsilon\delta^{d-1}}{d}\int_K |\bx-\by|^{1-d} v^2(\by)d\by
+\frac{\delta^{d+1}}{\epsilon d}\int_K |\bx-\by|^{1-d} |\nabla v(\by)|^2d\by.
$$
Now integrating both sides with respect to $\bx$ in $K$ gives
$$
|S|\int_K v^2dK\leq |K|\int_S v^2 dS +\frac{\epsilon\delta^{d}}{d}\int_K v^2dK
+\frac{\delta^{d+2}}{\epsilon d}\int_K |\nabla v|^2dK,
$$
which yields the desired estimate (\ref{aaa.08}) by setting $\epsilon=\frac{|S|d}{2\delta^d}$.
\end{proof}
\medskip

Consider a case of Lemma \ref{appendix-lemma2} in which the convex domain $K$ is a shape regular $d$-simplex. Denote by $h_K$ the diameter of $K$. The shape regularity implies that
\begin{enumerate}
\item the measure of $K$ is proportional to $h_K^d$,
\item there exists an inscribed ball $B_K\subset K$ with diameter proportional to $h_K$.
\end{enumerate}
Now let $S$ be a ball inside of $K$ with radius $r_S\ge \varsigma_* h_K$. Then, there exists a fixed constant $\kappa_*$ such that
\begin{equation}\label{aaa.21}
|K|\leq \kappa_* |S|.
\end{equation}
Apply (\ref{aaa.21}) in (\ref{aaa.08}) and notice that $\tau_* |S|\ge h_K^d$ and $\delta = h_K$. Thus,
\begin{equation}\label{aaa.28}
\|v\|^2_{K} \leq 2 \kappa_* \|v\|_{S}^2
+\frac{4\tau_*^2 h_K^{2}}{d^2}\|\nabla v\|_K^2.
\end{equation}
For simplicity of notation, we shall rewrite (\ref{aaa.28}) in the following form
\begin{equation}\label{aaa.38}
\|v\|^2_{K} \leq a_0 \|v\|_{S}^2+a_1 h_K^{2}\|\nabla v\|_K^2.
\end{equation}
If $v$ is infinitely smooth, then a recursive use of the estimate (\ref{aaa.38}) yields the following result
\begin{equation}\label{aaa.39}
\|v\|^2_{K} \leq \sum_{j=0}^n a_j h_K^{2j}\|\nabla^j v\|_{S}^2+a_{n+1} h_K^{2n+2}\|\nabla^{n+1} v\|_K^2.
\end{equation}
In particular, if $v$ is a polynomial of degree $n$, then
\begin{equation}\label{aaa.48}
\|v\|^2_{K} \leq \sum_{j=0}^n a_j h_K^{2j}\|\nabla^j v\|_{S}^2.
\end{equation}
The standard inverse inequality implies that
$$
\|\nabla^j v\|_{S} \lesssim h_K^{-j}\|v\|_S.
$$
Substituting the above into (\ref{aaa.48}) gives
\begin{equation}\label{aaa.58}
\|v\|^2_{K} \lesssim \|v\|_{S}^2.
\end{equation}
The result is summarized as follows.
\begin{lemma}[Domain Inverse Inequality]\label{appendix.thm}
Let $K\subset\mathbb{R}^d$ be a $d$-simplex which has diameter $h_K$ and is shape regular.
Assume that $S$ is a ball in $K$ with diameter $r_S$ proportional to $h_K$; i.e., $r_S\ge \varsigma_* h_K$ with a fixed $\varsigma_*>0$. Then, there exists a constant $C=C(\varsigma_*, n)$ such that
\begin{equation}\label{aaa.58-new}
\|v\|^2_{K} \le C(\varsigma_*,n) \|v\|_{S}^2
\end{equation}
for any polynomial $v$ of degree no more than $n$.
\end{lemma}

\medskip
The usual inverse inequality in finite element analysis also holds
true for piecewise polynomials defined on the finite element
partition $\T_h$ provided that it satisfies the assumptions {\bf
A1-A4} of Section \ref{wg-mfem}.

\begin{lemma}[Inverse Inequality]\label{appendix.inverse-inq}
Let $\T_h$ be a finite element partition of $\Omega$ consisting of
polygons or polyhedra. Assume that $\T_h$ satisfies all the
assumptions {\bf A1-A4} as specified in Section \ref{wg-mfem}. Then,
there exists a constant $C=C(n)$ such that
\begin{equation}\label{aaa.88-new}
\|\nabla \varphi\|_{T} \le C(n) h^{-1}_T \|\varphi\|_{T},\qquad
\forall T\in \T_h
\end{equation}
for any piecewise polynomial $\varphi$ of degree $n$ on $\T_h$.
\end{lemma}
\begin{proof}
The proof is merely a combination of Lemma \ref{appendix.thm} and
the standard inverse inequality on d-simplices. To this end, for any
$T\in \T_h$, let $S(T)$ be the circumscribed simplex that is shape
regular. It follows from the standard inverse inequality that
$$
\|\nabla \varphi\|_T\leq \|\nabla\varphi\|_{S(T)}\leq
Ch_T^{-1}\|\varphi\|_{S(T)}.
$$
Then we use the estimate (\ref{aaa.58-new}), with $K=S(T)$, to
obtain
$$
\|\nabla \varphi\|_T\leq Ch_T^{-1}\|\varphi\|_{S}\leq
Ch_T^{-1}\|\varphi\|_{T},
$$
where $S$ is a ball inside of $T$ with a diameter proportional to
$h_T$. This completes the proof of the lemma.
\end{proof}

\end{document}